\documentclass[11pt, oneside]{article}   	
\usepackage{geometry}                		
\usepackage{graphicx}				
\usepackage[utf8]{inputenc}
\usepackage[english]{babel}
\usepackage{cite}
\usepackage{fancyhdr}
\usepackage[affil-it]{authblk}

\usepackage{amssymb}
\usepackage{amsmath}
\usepackage{tikz-cd}
\usepackage{amsthm}
\newtheorem{theorem}{Theorem}[section]
\newtheorem*{theorem*}{Theorem}
\newtheorem{definition}{Definition}[section]

\newtheorem{proposition}{Proposition}[section]
\newtheorem{conjecture}{Conjecture}[section]
\theoremstyle{definition}
\newtheorem{remark}{Remark}[section]

\title{Wave Approximation of Backward Heat Equation with Ricci Flow}
\author{Jie Xu \thanks{Email Address: xujie@bu.edu}}
\affil{Boston University, Department of Mathematics and Statistics, Boston, Massachusetts, U.S.A.}

\date{}							

\begin{document}
\maketitle
\begin{abstract} In this paper, we consider solutions of the backward heat equation with Ricci flow on manifolds as a type of infinite dimensional limit of solutions of a wave equation on a larger manifold with an analysis of wavefront set. Specifically, the projection of the solution of the wave equation $ \left(\frac{2t}{N} \cdot \frac{\partial^{2}}{\partial t^{2}} + \frac{tR(t, x)}{N} \frac{\partial}{\partial t} - \Delta_{\tilde{g}^{(N)}(t)} \right) u = R(t, x) $ outside its wavefront set onto $ \mathbb{R}_{t} \times M_{x, g(t)} $ solves the backward heat equation $ \partial_{t} u + \Delta_{x, g(t)} u = -R(t, x) $ within some appropriate time interval. We discuss this approximation starting from Euclidean case, and then extend to the open Riemannian manifold situation. This idea partially comes from Perelman's original papers in proving Poincar\'e conjecture as well as Terence Tao's Notes in UCLA.
\end{abstract}

\noindent \section{Introduction}
It is great that Perelman proved the Poincar\'e conjecture by proving Thurston's geometrization conjecture by using Hamilton's Ricci flow tools. In particular, \cite{KL2} provided an excellent new definition as a generalization of Ricci flow in order to analyze the Ricci flow with surgery on 3 dimensional manifold better. They call it Ricci Flow Spacetime.
\begin{definition}
A Ricci Flow Spacetime is a tuple $ (\mathcal{M}, \mathfrak{t}, \partial_{t}, g) $ where:\\
(a) $ \mathcal{M} $ is a smooth manifold with boundary; \\
(b) $ \mathfrak{t} $ is the time function, a submersion $ \mathfrak{t} : \mathcal{M} \rightarrow I $, where $ I \subset \mathbb{R} $ is some time interval; \\
(c) The boundary of $ \mathcal{M} $, if it is not empty, must be denoted by $ \partial \mathcal{M} = \mathfrak{t}^{-1}(\partial I) $; \\
(d) $ \partial_{t} $ is the time vector field, which satisfies $ \partial_{t} \mathfrak{t} = 1 $; \\
(e) $ g $ is a smooth inner product in spatial subbundle $ \ker(d\mathfrak{t}) \subset \mathcal{TM} $, and $ g $ defines a Ricci flow $ \mathcal{L}_{\partial_{t}}g = -2Ric(g) $.
\end{definition}
We would like then to ask an interesting question: what is the specific map $ \mathfrak{t} : \mathcal{M} \rightarrow I $ that one could choose so that it provides the relationship between time evolution and the space, which indicates more information in geometry and topology of this Ricci Flow Spacetime. \\
\\
We try to consider Ricci flow as gradient flow, as Perelman stated in \cite{Per1}, \cite{Per3}, \cite{Per2}. With this setup, we can understand the non time-reversible, backward heat equation pairing with forward Ricci flow by variable coefficient wave equation, which is time-reversible. The main result for this paper is stated as follows. \\
\\
We study the Cauchy problem of variable coefficient wave equation
\begin{align}\label{one}
&\left(\frac{2t}{N} \cdot \frac{\partial^{2}}{\partial t^{2}} + \frac{tR(t, x)}{N} \frac{\partial}{\partial t} - \Delta_{\tilde{g}^{(N)}(t)} \right) u = R(t, x), (t, x, y) \in (0, T] \times M_{x, g(t)} \times \mathbb{R}_{y}^{N} \\
& u(T, x, y) = h(x), u_{t}(T, x, y) = -R(T, x). \nonumber
\end{align}
\noindent provided that $ h(x) \in \mathcal{C}_{0}^{\infty}(M_{x, g(t)}) $. Here $ M_{g(t)} $ is the noncompact Riemannian manifold $ (M, g(t)) $ without boundary satisfying Ricci flow $ \dot{g} = - 2Ric $; $ R(t, x) $ is the scalar curvature of $ (M, g(t)) $ at $ x \in M $; and the Riemannian metric $ \tilde{g}^{(N)}(t) $ on $  M_{x, g(t)} \times \mathbb{R}_{y}^{N} $ is given by
\begin{equation*}
\tilde{g}^{(N)}(t) = dy^{2} \oplus g(t).
\end{equation*}
\begin{theorem*}
Let $ Y : = (t,_{0}, T) \times \tilde{M}_{x, g(t)} \times B_{y}^{N} $ be any open, precompact submanifold of $ X : = [t_{0}, T] \times M_{x, g(t)} \times \mathbb{R}_{y}^{N} $ and the boundary of the closure of $ Y $ is $ \mathcal{C}^{1} $. Assume that Ricci flow exists at some time $ T > 0 $ and $ \text{supp}(R(t, x)) \in Y $ where $ R $ is the scalar curvature. \\
(i) for each $ N $, there exists a radial solution $ u = u^{(N)}(t, x, y) $ of (\ref{one}) with respect to $ y $ - coordinates, i.e. $ u^{(N)}(t, x, y) = u^{(N)}(t, x, Ay), \forall A \in SO(N) $, such that the limit $ u(t, x) : = \lim_{N \rightarrow \infty} u^{(N)}(t, x, y) |_{ \lbrace t = \lvert y \rvert^{2} \slash 2N \rbrace} $ on $ [t_{0}, T] \times \tilde{M}_{x, g(t)} $ exists outside points $ \mathcal{P} $ in the projection of wavefront set $ WF_{\infty} $ onto $ (0, T] \times M_{x, g(t)} $. 
\par
(ii) Outside points in $ \mathcal{P} $, the limit $ u(t, x) \in [t_{0}, T] \times \tilde{M}_{x, g(t)} $ solves the backward heat equation below
\begin{equation}
\partial_{t} u + \Delta_{x, g(t)} u = -R, u(T, x) = h(x), (t, x) \in (0, T] \times M_{x, g(t)}.
\end{equation}
\end{theorem*}
\noindent The definition of $ WF_{\infty} $ will be defined below, which corresponds to the wavefront set of (\ref{one}) for each $ N $.

\noindent \section{Motivation}
In this section, we describe Perelman's idea about using Hamilton's Ricci flow to prove Thurston's geometrization conjecture among 3 dimensional manifolds. \\
\par
In Perelman's original papers and other notes and related follow-ups of Perelman's work \cite{KL}, \cite{MT}, \cite{CZ}, especially at Section 6 of \cite{Per1} and \cite{Tao}, \cite{Tao9} notes, they explained a way of converting some naturally established backward heat equation to almost harmonic equation when embedded original manifold $ (M, g(t)) $ into a larger space $ (M \times \mathbb{R}^{N} \times \mathbb{R}^{-}, \tilde{g}^{(n)}(t)): = (\tilde{M}^{(N)}, \tilde{g}^{(N)}(t)) $. \\
\par
In fact, the converted almost harmonic equation is exactly the wave equation with variable coefficient. \cite{Per1}, \cite{Per3}, \cite{Per2}, \cite{CZ}, \cite{Tao} mentioned that in the enlarged space $ \tilde{M}^{(N)} $ with some appropriated choice of Riemannian metric, the enlarged space is almost Ricci flat. This one indicates the validity of Bishop-Gromov inequality at least when $ N \rightarrow \infty $, which follows that the Bishop-Gromov reduced volume is nonincreasing when the radius of the ball increases. \cite{Per1}, \cite{Per3}, \cite{Per2} then introduced his famous Perelman reduced length and Perelman reduced volume, which has some monotonicity property, a mimic of Bishop-Gromov reduced volume. It has been shown to play an important role in the proof of Thurston's geometrization conjecture in 3 dimensional manifold. \\
\par
The heat equation that \cite{Per1}, \cite{Tao9} constructed relates to the understanding of Ricci flow as a gradient flow, which is naturally corresponding to some heat equation backward in time. The solution of this heat equation associates to the evolution of $ g(t) $ with Ricci flow. The wave approximation with some good choice of Riemannian metric in the enlarged space is now not only disclosing the Ricci flat property as well as the consequences above, but the solution of the heat equation or its equivalent wave equation reveals the evolution of the Ricci flow itself. \\
\par
It follows that if we could understand the solution of the wave approximation of the heat equation, we understand the solution of the heat equation itself on some subset of the enlarged space; furthermore, it provided us a chance to detect the singularity of the Ricci flow via this wave approximation, according to the detailed analysis of the wavefront set of this variable coefficient wave equation. \\
\par
The rest part of this section is to mathematically describe the picture above, the major credits are due to Tao's notes \cite{Tao} as well Perelman's original paper. 
\begin{remark}
Note that we will use $ \partial_{t} g = \dot{g} $ interchangeably without further notice. Note also that we will freely use either $ d\mu(t) = d\mu $ or $ dvol_{g} $ as the volume form of the Riemannian manifold $ (M, g(t)) $, we will emphasize our preference in some specific situation once it is needed.
\end{remark}
For simplicity, we assume the evolution of Riemannian manifolds to be $ (M, g(t)) $, where $ M $ is some compact manifold with $ \dim M = n $ such that the evolution of Riemannian metric satisfies Ricci flow
\begin{equation}
\partial_{t}g = -2Ric.
\end{equation}
\noindent Consider the static metric $ (M, g) $ where $ M $ is compact at the time being. We know that the heat equation could be understood as a gradient flow with the following way. For one side, if we assume $ f: [0, \infty) \times M \rightarrow R $ solves the heat equation $ \partial_{t} f = \Delta_{g} f $, then $ f $, as a scalar field, is considered to be the gradient flow of the Dirichlet energy
\begin{equation*}
E(f) = \frac{1}{2} \int_{M} \lvert \nabla f \rvert_{g}^{2} dvol_{g}.
\end{equation*}
\noindent If we choose the inner product on $ M $ is given by $ \langle f, g \rangle = \int_{M} fg dvol_{g} $, then we can see that in the sense of Fr\'echet derivative
\begin{align*}
E(f + h) - E(f) & = \frac{1}{2} \int_{M} g(\nabla (f + h), \nabla (f + h)) dvol_{g} - \frac{1}{2} \int_{M} \lvert \nabla f \rvert_{g}^{2} dvol_{g} \\
& = \int_{M} g(\nabla f, \nabla h) dvol_{g} + \frac{1}{2} \int_{M} \lvert \nabla h \rvert_{g}^{2} dvol_{g} = -\int_{M} h \Delta_{g} f dvol_{g} + o(h),
\end{align*}
\noindent provided that $ f, h $ has enough regularity. Hence it follows that if $ f $ solves the gradient flow, then in the sense of distribution
\begin{equation*}
\langle -\nabla E(f), v \rangle = \int_{M} f \Delta_{g} v dvol_{g} = \langle \Delta_{g} f, v \rangle = \langle \partial_{t}f, v \rangle.
\end{equation*}
\noindent It follows that $ \partial_{t} f = -\nabla E(f) $ in the distributional sense and hence $ f $ is a gradient flow of the Dirichlet energy. The other direction is similar and hence we conclude that solution of heat equation is equivalent to solution of gradient flow with some appropriate choice of energy function.
\begin{remark}
One can observe that since
\begin{equation*}
\frac{d}{dt} E(f) = \frac{d}{dt} \left( \frac{1}{2} \int_{M} \lvert \nabla f \rvert_{g}^{2} dvol_{g}\right) = \int_{M} - \lvert \Delta_{g} f \rvert^{2} dvol_{g} = -\langle \Delta_{g} f, \partial_{t} f \rangle.
\end{equation*}
\noindent The integrand with in the time derivative of the energy function is the expression of $ \partial_{t} f $, when $ f $ follows the gradient flow with respect to $ E(f) $. We will use this interpretation below.
\end{remark}
\begin{remark}
Among the derivation above, we apply integration by parts on compact manifold due to the fact that $ \int_{M} \text{div}(X) d\mu = 0 $, for vector field $ X $ as section of the tangent bundle. We can definitely extend the same integration by parts to noncompact manifold with the integrand compactly supported.
\end{remark}
With this observation, we are considering the Ricci flow on $ (M, g(t)) $ where $ M $ is compact, $ t \in [0, \infty) $. As \cite{Per1}, \cite{Per2}, \cite{Tao} mentioned, we are trying to consider the Ricci flow as a gradient flow, for which we could obtain one more quantity control for the evolution of Riemannian metric. \\
\\
Following Tao's notation in \cite{Tao}, \cite{Tao9}, we define $ \Delta_{L} $, the Hodge De-Rham Laplacian on $ (0, 2) $ - tensor $ \pi_{\alpha \beta} $ to be
\begin{equation}\label{two}
\Delta_{L} \pi_{\alpha \beta} = \Delta_{g} \pi_{\alpha \beta} + 2g^{\sigma \gamma} \text{Riem}_{\sigma \alpha \beta}^{\delta} \pi_{\gamma \delta} - \text{Ric}_{\alpha}^{\gamma} \pi_{\gamma \beta} - \text{Ric}_{\beta}^{\gamma} \pi_{\gamma \alpha}.
\end{equation}
\noindent where $ \Delta_{g} \pi_{\alpha \beta} = \nabla_{\gamma} \nabla^{\gamma} \pi_{\alpha \beta} $ is the usual connection Laplacian. Note that $ \nabla^{\gamma} $ is lifted from $ \nabla_{\beta} $ by $ g^{\beta \gamma} $. \\
\\
\noindent Apply (\ref{two}), and take the trace of the following equation
\begin{equation*}
\dot{Ric}_{\alpha \beta} = - \frac{1}{2} \Delta_{L} \dot{g}_{\alpha \beta} - \frac{1}{2} \nabla_{\alpha} \nabla_{\beta} \text{tr}(\dot{g}) + \frac{1}{2} \nabla_{\alpha} \nabla^{\gamma} \dot{g}_{\beta \gamma} + \frac{1}{2} \nabla_{\beta} \nabla^{\gamma} \dot{g}_{\alpha \gamma}.
\end{equation*}
\noindent We obtain the first variation formula for the scalar curvature
\begin{equation}
\dot{R} = -Ric^{\alpha \beta} \dot{g}_{\alpha \beta} - \Delta_{g} \text{tr}(\dot{g}) + \nabla^{\alpha} \nabla^{\beta} \dot{g}_{\alpha \beta}.
\end{equation}
\begin{remark}
Note that here $ g $ is evolving when time changes. Furthermore, in order to mimic the interpretation of gradient flow in Remark 2.2, we define a version of inner product on symmetric $ (0,2) $ - forms on $ (M, g) $ as
\begin{equation}
\langle u, v \rangle_{g} = \int_{M} u^{\alpha \beta} v_{\alpha \beta} dvol_{g}.
\end{equation}
\end{remark}
\noindent \cite{KL}, \cite{Tao} mentioned that based on the variational formula, we would like to see whether the Einstein-Hilbert functional
\begin{align}
& H(M, g) = \int_{M} R d\mu \\
\Rightarrow & \int_{M} \dot{R} d\mu = \int_{M}  -Ric^{\alpha \beta} \dot{g}_{\alpha \beta} - \Delta_{g} \text{tr}(\dot{g}) + \nabla^{\alpha} \nabla^{\beta} \dot{g}_{\alpha \beta} d\mu \nonumber \\
= & \int_{M} - Ric^{\alpha \beta} \dot{g}_{\alpha \beta} d\mu, \; \text{by Stokes' Theorem $ \int_{M} \text{div}(X) d\mu = 0 $} \nonumber
\end{align}
\noindent is the energy of some gradient flow. Here we use the following formula
\begin{equation}\label{three}
\partial_{t} d\mu(t) = \frac{1}{2} \text{tr}(\dot{g}) d\mu(t).
\end{equation}
\noindent which is obtained from the local form $ d\mu(t) = \sqrt{\det g(t)}dx^{1} \wedge \dotso \wedge dx^{n} $. Thus using the interpretation in Remark 2.2, it seems that the Ricci flow $ \dot{g} = -2Ric $ could be considered as the gradient flow of the energy $ -2H $. But since by (5), $ d \mu $ is also a function of $ t $, so we can compute from (4) that
\begin{align*}
\frac{d}{dt} \int_{M} R d\mu(t) & = \int_{M} \dot{R} d\mu(t) + \int_{M} R \frac{d}{dt} d\mu(t) \\
& = \int_{M} (- Ric^{\alpha \beta} \dot{g}_{\alpha \beta} + \frac{1}{2} R \text{tr}(\dot{g})) d\mu(t) \\
& = \int_{M} (- Ric^{\alpha \beta} + \frac{1}{2} R g^{\alpha \beta}) \dot{g}_{\alpha \beta} d\mu(t). 
\end{align*}
\noindent Here the last term is obtained by using the definition of $ \text{tr}(\dot{g}) $. From this, we see that the Ricci flow is not the the gradient flow with respect to the energy $ H(M, g) $, the Einstein-Hilbert functional. However, due to the interpretation in Remark 2.2 and the inner product (4), we conclude that $ - 2H $ corresponds to the gradient flow
\begin{equation*}
\dot{g} = -2Ric + Rg.
\end{equation*}
\noindent where $ G = -Ric + \frac{1}{2} Rg $ is the Einstein tensor. The issue that makes Ricci flow not to be the gradient flow of Einstein-Hilbert functional is due to the variation of volume form $ d\mu(t) = dvol_{g(t)} $ under time evolution. To fix this problem, we introduce some static metric $ g_{0} $ with volume form $ dvol_{g_{0}} = dm $.\\
\\
By Radon-Nikodym theorem, there exists a positive function, which could be denoted by $ e^{-f} $, where $ f $ is also time-dependent, such that
\begin{equation}\label{four}
dm = e^{-f} d\mu(t).
\end{equation}
\noindent A quick consequence of equation (\ref{four}) is obtained by taking $ \partial_{t} $ on both sides, since $ \frac{d}{dt} dm = 0 $, then using (\ref{three}), we have
\begin{equation}
\frac{d}{dt} d\mu(t) = \dot{f} e^{f} dm \Rightarrow \frac{1}{2} \text{tr}(\dot{g}) d\mu(t) = \dot{f} d\mu(t) \Rightarrow \partial_{t} f = \frac{1}{2} \text{tr}(\dot{g}).
\end{equation}
\noindent With this relation of Radon-Nikodym derivative, our next try is to find out gradient flow with respect to $ \int_{M} R dm $, and unfortunately, it is not connecting to the Ricci flow. \\
\\
In Perelman's paper and Tao's notes \cite{Per1}, \cite{Per3}, \cite{Per2}, \cite{Tao}, they suggested the following functional
\begin{equation}
\mathcal{F}_{m}(M, g) = \int_{M} (\lvert \nabla f \rvert_{g}^{2} + R) dm.
\end{equation}
\noindent which is the summation of Einstein-Hilbert functional and a multiple of Dirichlet energy functional.
\begin{remark} It is a good place to remark that we took $ t $ - derivative with respect to $ H(M, g) $, and will take it on $ \mathcal{F}_{m}(M, g) $, both request the exchange with the integral sign, they are guaranteed by dominated convergence theorem since all integrands are smooth under a compact manifold.
\end{remark}
We can compute that
\begin{align}
\frac{d}{dt} \int_{M} R dm & = \int_{M} (-Ric^{\alpha \beta} \dot{g}_{\alpha \beta} - \Delta_{g} \text{tr}(\dot{g}) + \nabla^{\alpha} \nabla^{\beta} \dot{g}_{\alpha \beta}) dm \nonumber \\
& = \langle -Ric^{\alpha \beta} - (\lvert \nabla f \rvert_{g}^{2} - \Delta_{g} f)g^{\alpha \beta} + (\nabla^{\alpha}f)(\nabla^{\beta}f) - \nabla^{\alpha} \nabla^{\beta}f, \dot{g}_{\alpha \beta} \rangle_{m}. \label{five}\\
\frac{d}{dt} \int_{M} \lvert \nabla f \rvert_{g}^{2} dm & = \langle -\Delta_{g}f g^{\alpha \beta} + \lvert \nabla f \rvert_{g}^{2} g^{\alpha \beta} - (\nabla^{\alpha} f)(\nabla^{\beta} f), \dot{g}_{\alpha \beta} \rangle_{m}. \label{six}
\end{align}
\noindent Combining (\ref{five}) and (\ref{six}), we have
\begin{equation}\label{seven}
\frac{d}{dt} \mathcal{F}_{m}(M, g) = \int_{M} (-Ric^{\alpha \beta} - \nabla^{\alpha} \nabla^{\beta} f)\dot{g}_{\alpha \beta} dm.
\end{equation}
\noindent Due to the inner product of the form (6), (\ref{seven}) indicates that the gradient flow with respect to the energy $ - 2\mathcal{F}_{m} (M, g) $ is given by
\begin{equation}\label{eight}
\dot{g}_{\alpha \beta} = -2Ric_{\alpha \beta} - 2\nabla_{\alpha} \nabla_{\beta} f.
\end{equation}
\noindent Taking the trace of (\ref{eight}), we conclude that
\begin{align}
& g^{\alpha \beta} \dot{g}_{\alpha \beta} = -2 g^{\alpha \beta} Ric_{\alpha \beta} - 2g^{\alpha \beta} \nabla_{\alpha} \nabla_{\beta} f; \nonumber \\
\Rightarrow & \text{tr}(\dot{g}) = - 2R - 2\Delta_{g} f; \nonumber \\
\Rightarrow & \partial_{t}f = - R - \Delta_{g} f. \; \text{Using (10)} \label{nine}
\end{align}
\noindent It follows that the gradient flow with respect to the energy $ - 2\mathcal{F}_{m} (M, g) $ is almost Ricci flow with an extra term with respect to $ f $. \cite{Per1}, \cite{Tao9}, and Section 9 of \cite{KL} discussed how to consider this as a family of gradient flow modulo diffeomorphism by observing that $ \nabla_{\alpha} \nabla_{\beta} f = \frac{1}{2} \mathcal{L}_{\nabla f} g_{\alpha \beta} $.
\begin{remark}
It follows that equation (\ref{nine}), $ \partial_{t} f + \Delta_{g}f = -R $ evolving $ f $ is natural backward in time, which can be seen as follows. Substituting (\ref{eight}) into (\ref{seven}), the energy function $ \mathcal{F}_{m}(M, g) $ with respect to this heat equation has the property
\begin{align*}
\frac{d}{dt} \mathcal{F}_{m}(M, g) & = \int_{M} (-Ric^{\alpha \beta} - \nabla^{\alpha} \nabla^{\beta} f)\dot{g}_{\alpha \beta} dm \\
& = \int_{M} (-Ric^{\alpha \beta} - \nabla^{\alpha} \nabla^{\beta} f)(-2Ric_{\alpha \beta} - 2\nabla_{\alpha} \nabla_{\beta} f) dm \\
& = 2\int_{M} \lvert Ric + Hess(f) \rvert_{g}^{2} dm,
\end{align*}
\noindent which can be seen is nondecreasing when forward in time. 
Combining this observation with the Ricci flow, we can solve this heat equation backward on any time interval $ [t_{1}, t_{2}] \in [0, \infty) $ provided that the Ricci flow has solution at $ t_{2} $.
\end{remark}
\begin{remark}
Recall that $ f $ is given by the Radon-Nikodym derivative by $ dm = e^{-f} d\mu(t) $ or $ d\mu(t) = e^{f} dm $, analysis of this backward heat equation indicates the evolution of the volume form backward in time.\end{remark}
\begin{remark} In Ricci flow, It seems plausible from formula (10) that $ \partial_{t} f = \frac{1}{2} \text{tr}(\dot{g}) = - R $. Substitute this into equation (\ref{nine}), we see that (\ref{nine}) is equivalent to consider
\begin{equation*}
\Delta_{g(t)} f = 0.
\end{equation*}
\noindent So the solution of the backward equation seems corresponds to the solution of the elliptic equation, and it seems follows that $ f \equiv c $ for some constant $ c $ since $ M $ is a compact manifold. It contradicts to the evolution of the Ricci flow, especially the examples of gradient shrinking solitons as Gaussian shrinking soliton, shrinking round sphere, etc. So where is the problem coming from? \\
\par
The issue is that the Ricci flow may not exist everywhere in $ M $ at every time $ t $. Thus the Laplace equation $ \Delta_{g(t)}f = 0 $ may only be defined at some subset $ U $ of $ M $ and thus the solution of $ \Delta_{g(t)}f = 0 $ is highly characterized by topology and geometry of $ U $ , especially the corresponding boundary conditions. {\bf This leads us to consider how the subset is looking like.}
\end{remark}
\begin{remark}
On any time interval $ [t_{1}, t_{2}] \subset [0, \infty) $, we can choose $ dm $ to be $ dvol_{g(t_{2})} $ (provided that $ g(t_{2}) $ does exist on Ricci flow), and hence $ dm = d\mu(t_{2}) $ which implies that $ e^{f} = 1 $ at $ t = t_{2} $. From now on, we consider the Cauchy problem (14) with the initial condition $ f(t_{2}, x) = 0 $ on $ [t_{1}, t_{2}] \times M $. \\
\\
If we consider the change of variables by $ \tau = t_{2} - t $, then we consider the backward heat equation $ \partial_{\tau}f(t_{2} - \tau, x) = \pm \Delta_{g(t_{2} - \tau)} f(t_{2} - \tau, x) - R $ with $ f(0, x) = 0, \tau \in [0, t_{2} - t_{1}] $.
\end{remark}

\noindent \section{The Analysis of Wave Approximation of Backward Heat Equation on Ricci Flow: Part I}
Let's take a look at the backward equation (\ref{nine}) again. Firstly we can set $ t_{1} = 0, t_{2} = T $ and hence we consider the pair of forward Ricci flow and backward heat equation
\begin{equation}
\begin{cases} & \dot{g} = - 2Ric \\ & \partial_{t} f + \Delta_{g}f = - R \end{cases},
\end{equation}
\noindent on $ [0, T] \subset [0, \infty) $ provided that $ \partial_{t}f |_{t= T} = -R(T) $ is equal to some function, under the Ricci flow. However, it is an issue that neither of the equations is time-reversible. Thus the solvability of this system relies on the assumption that either the Ricci flow exists at some later time or the backward heat equation has initial condition at some earlier time and flows forward in time; however, both assumptions cannot be achieved simultaneously. This inspired us to consider a time-reversible expression of (\ref{nine}), which is (\ref{one}) above. \\
\\
All we can know or construct or set-up is the initial data $ (M, g(0)) $. It follows that we would like to understand the Ricci flow as well as equation (\ref{nine}) forward in time, and obviously the backward heat flow cannot support this. It inspired us to find a good way to convert the backward heat flow to some new PDE that describes the same behavior, and the new behavior should be time reversible so that we can use the information at initial status, $ (M, g(0)) $. \cite{Per1} indicated this in his paper with the idea below. \\
\\
The point that \cite{Per1}, \cite{Per3}, \cite{Per2} indicated, which also mentioned by \cite{CZ}, that the genius of Grisha's work is to embed the space $ \mathbb{R}^{+} \times M_{x} $ into the space $ \mathbb{R}^{+} \times M_{x} \times \mathbb{S}_{y}^{N} : = \mathbb{R}^{+} \times \tilde{M} $ with new Riemmanian metric
\begin{equation}\label{ten}
\tilde{g}_{ij} = g_{ij}, \tilde{g}_{00} = \frac{N}{2\tau} + R, \tilde{g}_{\alpha \beta} = \tau g_{\alpha \beta}, \tilde{g}_{0\alpha} = \tilde{g}_{0i} = \tilde{g}_{\beta j} = 0; 
\end{equation}
\noindent Here $ i, j $ are indices on $ M_{x} $, the original manifold, labeled by local coordinates in $ x $; $ \alpha, \beta $ are indices on $ \mathbb{S}_{y}^{N} $ with curvature $ \frac{1}{2N} $, labeled by local coordinates $ y $, and the coordinate $ \tau $ on $ \mathbb{R}^{+} $ has index $ 0 $. \\
\\
\cite{Per1}, \cite{CZ} mentioned that here $ \tau = T - t $ for the $ T $ mentioned above, and hence we consider the backward Ricci flow $ \partial_{\tau} g = 2Ric $ on the same time interval $ [0, T] $. Note that if we make this change to the backward Ricci flow, then equation (10) becomes $ \partial_{\tau} f = -\frac{1}{2} \text{tr}(\dot{g}) $ and it follows that equation (\ref{nine}) becomes
\begin{equation*}
- R - \Delta_{g} f = \frac{1}{2} \text{tr}(\dot{g}) = - \partial_{\tau} f.
\end{equation*}
\noindent which is now a standard heat equation forward in time interval $ [0, T] $.
\begin{remark}
In this situation, we again somehow loss the information at the beginning time. Alongside the backward Ricci flow starting at some time $ \tau = \tau_{0} > 0 $, we are not sure whether the Ricci flow will last until, for example, $ \tau = 0 $. However, again, in order to solve the forward heat equation, we must assume that $ f $ does exist at $ \tau = 0 $.
\end{remark} 
\begin{remark} I wonder whether Perelman's setup is absolutely correct. He mentioned in his paper that "the heat equation and conjugate heat equation on the original space $ M $ can be interpreted as almost Laplace equation on $ \tilde{M} $ modulo $ N^{-1} $". From Tao's observation in \cite{Tao} as well as his description, I believe that his backward Ricci flow must be on the space $ \mathbb{R}^{-} \times \tilde{M} $ by shifting $ \tau $ by $ T $ to the left. And his Riemannian metric is with respect to this interpretation of $ \tau = -t \leqslant 0 $ on $ [-T, 0] $ sitting in the space $ \mathbb{R}^{-} \times \tilde{M} $.
\end{remark}
We move on one step from Remark 3.2 relying on (cf. Tao's ) great observation that the Riemannian metric above could be obtained by setting $ \tau =- \frac{\lvert y \rvert^{2}}{2N} $ provided that the Riemmannian metric we imposed on the new space $ \mathbb{R}^{-} \times \tilde{M} $ as follows
\begin{equation}\label{eleven}
\tilde{g}{(N)} = dy^{2} + \frac{r^{2}}{N^{2}} R(t) dr^{2} +g(t), r^{2} = \sum_{i=1}^{N} y_{i}^{2}.
\end{equation}
\noindent Using the change of variables $ \tau = -\frac{\lvert y \rvert^{2}}{2N} $ , we can compute that
\begin{equation*}
d\tau^{2} = d(-\frac{r^{2}}{2N})^{2} = (-\frac{r}{N}dr)^{2} = \frac{r^{2}}{N^{2}} dr^{2} = \frac{-2\tau}{N} dr^{2} \Rightarrow dr^{2} = \frac{N}{-2\tau} d\tau^{2} = \frac{N^{2}}{r^{2}} dr^{2}.
\end{equation*}
\noindent Plug in this to (\ref{eleven}), we observe that
\begin{align}
\tilde{g}{(N)} & = dy^{2} + \frac{r^{2}}{N^{2}} R(t) dr^{2} +g(t) = dr^{2} + r^{2}d\omega_{1}^{2} + \frac{r^{2}}{N^{2}} R(t) dr^{2} +g(t) \nonumber \\
& = \frac{N}{-2\tau} d\tau^{2} + (-2N\tau) d\omega_{1}^{2} + R(t) d\tau^{2} + g(t) \nonumber \\
& = (\frac{N}{-2\tau} + R(t)) d\tau^{2} + (-\tau) d\omega_{1 \slash 2N}^{2} + g(t). \label{twelve}
\end{align}
\noindent Which is the same as the Riemannian metric Grisha/Grigori have chosen in (\ref{ten}) since $ \tau = -\lvert y \rvert^{2} \slash 2N \leqslant 0 $. Note that $ \omega_{\kappa} $ denotes the round metric on sphere with curvature $ \kappa $.
\begin{remark}
It requires a remark here to explain what I mean by THE SAME. Note that in Grisha's setup, he said the enlarged space is $ \mathbb{R}^{+} \times \tilde{M} $, then one way to verify this is that people should think his $ \tau = T - t $, where we should set $ \tau = T - \frac{\lvert y \rvert^{2}}{2N} $, in this case, (\ref{twelve}) is of the form
\begin{equation*}
\tilde{g}{(N)} = (\frac{N}{2(T - \tau)} + R(t)) d\tau^{2} + (T-\tau) d\omega_{1 \slash 2N}^{2} + dg(t),
\end{equation*}
\noindent which is good to consider the backward Ricci flow as well as the forward heat equation between $ [0, T] $. But note that here $ R $ is a function of $ t $ and $ g $ is evolving alongside $ t $ and $ t = T - \tau $. So if we do this change of variable, (\ref{twelve}) in $ t $ is of the form
\begin{equation}
\tilde{g}{(N)} = (\frac{N}{2t} + R(t)) dt^{2} + t d\omega_{1 \slash 2N}^{2} + g(t).
\end{equation}
\noindent Hence I believe that what Grisha meant is that we embed $ \mathbb{R}_{t}^{+} \times M_{x} \hookrightarrow \mathbb{R}_{t}^{+} \times M_{x} \times \mathbb{S}_{y}^{N} $ in the sense of forward Ricci flow and backward heat equation with the Riemannian metric (\ref{ten}).
\end{remark}
\noindent In Section 6 of Perelman's paper \cite{Per1}, he stated the following statement without proof:
\begin{conjecture} The heat equation and conjugate heat equation on the original space $ M $ can be interpreted as almost Laplace equation on $ \tilde{M} $ modulo $ N^{-1} $ in the sense that $ u $ satisfies the heat equation on $ \mathbb{R}^{+} \times M $ if and only if $ \tilde{u} $ satisfies $ \tilde{\Delta} \tilde{u} = 0 $ modulo $ N^{-1} $ on $ \mathbb{R}^{+} \times \tilde{M} $ provided that $ \tilde{u} $, or equivalently being considered as the extension of $ u $ to $ \mathbb{R}^{+} \times \tilde{M} $ is constant along the $ \mathbb{S}^{N} $ fibers; simiarly, u satisfies the conjugate heat equation on $ \mathbb{R}^{+} \times M $ if and only if $ \tilde{u}^{*} = r^{-\frac{N-1}{2}}\tilde{u} $ satisfies $ \tilde{\Delta} \tilde{u}^{*} = 0 $ modulo $ N^{-1} $ on $ \mathbb{R}^{+} \times \tilde{M} $.
\end{conjecture}
Our purpose is to provide the explicit expression of this almost Laplacian equation in the backward heat equation case, show that it is really a variable coefficient wave equation, which is time reversible. The advantage of time reversible property has been explained above. We will construct a toy example, proof Perelman's conjecture \cite{Per1} in this toy case, and then extend the same proof to the backward heat equation $ \partial_{t}f = - \Delta_{g}f - R $ where $ g $ follows the forward Ricci flow between $ (0, T] $. This equivalence does not only provide the physical observation that infinite speed of propagation associated with heat equation can be approximated by the increasing sequence of finite speed of propagations with wave flow, but also provide us the propagation of singularities that can be revealed by wavefront set of the wave equation. The wavefront set explained both the positions that the solution has not enough regularity and the directions that the solution cannot decay rapidly/directions that the solution do not have enough regularity. (which leads to the introduction of Perelman reduced length and Perelman reduced volume, which has some monotonicity property; as a mimic of Bishop-Gromov reduced volume and Bishop-Gromov volume comparison theorem, Perelman proved the similar comparison theorem for Perelman reduced volume, and leads to the introduction of $ \kappa $ - solution whose asymptotic behavior is gradient shrinking soliton. He then use this to classify the 3 dimensional manifolds.) \\ {
\\
To simplify the problem a little bit, we apply the change of variable $ t = \frac{\lvert y \rvert^{2}}{2N} $ with respect to a slightly simpler metric $ \tilde{g}^{(N)}(t) = dy^{2} + g(t) $ and observe that the Laplacian of $ \tilde{\Delta}_{\tilde{g}^{(N)}} \tilde{u} = 0 $ modulo $ N^{-1} $ is exactly of the form
\begin{equation}\label{thirteen}
\tilde{\Delta}_{\tilde{g}^{(N)}} = \Delta_{y} + \Delta_{x, g(t)} - \frac{t \cdot \text{tr}(\dot{g})}{N}\partial_{t}
\end{equation}
\noindent From this point, we study the Cauchy problem of variable coefficient wave equation
\begin{align}\label{oneone}
&\left(\frac{2t}{N} \cdot \frac{\partial^{2}}{\partial t^{2}} + \frac{tR(t, x)}{N} \frac{\partial}{\partial t} - \Delta_{\tilde{g}^{(N)}(t)} \right) u = R(t, x), (t, x, y) \in (0, T] \times M_{x, g(t)} \times \mathbb{R}_{y}^{N} \\
& u(T, x, y) = h(x), u_{t}(T, x, y) = -R(T, x). \nonumber
\end{align}
\noindent provided that $ h(x) \in \mathcal{C}_{0}^{\infty}(M_{x, g(t)}) $. Here $ M_{g(t)} $ is the noncompact Riemannian manifold $ (M, g(t)) $ without boundary satisfying Ricci flow $ \dot{g} = - 2Ric $; $ R(t, x) $ is the scalar curvature of $ (M, g(t)) $ at $ x \in M $; and the Riemannian metric $ \tilde{g}^{(N)}(t) $ on $  M_{x, g(t)} \times \mathbb{R}_{y}^{N} $ is given by
\begin{equation*}
\tilde{g}^{(N)}(t) = dy^{2} \oplus g(t).
\end{equation*}

\begin{theorem}
Let $ Y : = (t,_{0}, T) \times \tilde{M}_{x, g(t)} \times B_{y}^{N} $ be any open, precompact submanifold of $ X : = [t_{0}, T] \times M_{x, g(t)} \times \mathbb{R}_{y}^{N} $, $ 0 \in B_{y}^{N} $ and the boundary of the closure of $ Y $ is $ \mathcal{C}^{1} $. Assume that Ricci flow exists at some time $ T > 0 $ and $ \text{supp}(R(t, x)) \in Y $ where $ R $ is the scalar curvature. \\
(i) for each $ N $, there exists a radial solution $ u = u^{(N)}(t, x, y) $ of (\ref{oneone}) with respect to $ y $ - coordinates, i.e. $ u^{(N)}(t, x, y) = u^{(N)}(t, x, Ay), \forall A \in SO(N) $, such that the limit $ u(t, x) : = \lim_{N \rightarrow \infty} u^{(N)}(t, x, y) |_{ \lbrace t = \lvert y \rvert^{2} \slash 2N \rbrace} $ on $ [t_{0}, T] \times \tilde{M}_{x, g(t)} $ exists outside points $ \mathcal{P} $ in the projection of wavefront set $ WF_{\infty} $ onto $ [t_{0}, T] \times \tilde{M}_{x, g(t)} $. 
\par
(ii) Outside points in $ \mathcal{P} $, the limit $ u(t, x) \in [t_{0}, T] \times \tilde{M}_{x, g(t)} $ solves the backward heat equation below
\begin{equation}\label{heat}
\partial_{t} u + \Delta_{x, g(t)} u = -R, u(T, x) = h(x), (t, x) \in (0, T] \times M_{x, g(t)}.
\end{equation}
\end{theorem}
\noindent Briefly speaking, we set $ WF_{\infty} : = \lim_{N \rightarrow \infty} WF(u^{(N)}) $ to be the set of essential singularities. The metric we use in $ WF_{\infty} $ is the product of metrics on $ \mathbb{R}_{t} $, Riemannian metric on $ M $ and the $ l^{2} $ - metric on $ \bigoplus_{i=1}^{\infty} \mathbb{R} $, the direct sum of countably many copies of $ \mathbb{R} $. The points in $ WF_{\infty} $ are either arbitrarily close to points in $ WF(u^{(N)}) $ for all large $ N $ by embedding $ WF(u^{(N)}) $ into the space of $ \mathbb{R}_{t}^{2} \times T^{*}M \times (\bigoplus_{i=1}^{\infty} \mathbb{R})^{2} $ or with arbitrarily large norms for all large $ N $.
\begin{remark} We will prove the Euclidean version of this theorem below in Section 4, and return to the proof of this theorem in Section 5. The standard definition of $ WF_{\infty} $ is also in Section 4 below.
\end{remark}

\noindent \section{The Analysis of Wave Approximation of Backward Heat Equation: Euclidean Case}
As the toy example, we consider the homogeneous backward heat equation
\begin{equation}\label{fourteen}
\partial_{t} u + \Delta_{x} u = 0, u: (0, T] \times \mathbb{R}_{x}^{n} \rightarrow \mathbb{R}, u(T, x) = h(x),
\end{equation}
\noindent with the standard Euclidean metric and $ h $ has enough regularity so that we can have $ u \in \mathcal{C}^{2}((0, T] \times \mathbb{R}^{n}) $.

\begin{remark} The Euclidean version (\ref{fourteen}) above is obtained by reducing the time-dependent Riemannian manifold $ (M,g(t)) $ into $ (\mathbb{R}^{n}, dx^{2}) $, independent of $ t $; correspondingly, $ R(t, x) \equiv 0, \forall t, x $. Therefore the enlarged space $ \tilde{M} $ above is exactly $ \mathbb{R}_{x}^{n} \times\mathbb{R}_{y}^{N} $ with the Riemannian metric $ dx^{2} + dy^{2} $.
\end{remark}
Clearly this backward heat equation has solution backward in time. One way to check is that if we assume the equation within some compact manifold $ M $ instead of $ \mathbb{R}^{n} $, then we can see that
\begin{equation*}
\frac{d}{dt} \int_{M} \lvert \nabla u \rvert_{g}^{2} dvol_{g} = \int_{M} g(\nabla \dot{f}, \nabla f) dvol_{g} = \int_{M} -\dot{f} \Delta_{g}f dvol_{g} = \int_{M} \lvert \Delta_{g} f \rvert^{2} dvol_{g} \geqslant 0.
\end{equation*}
\noindent So when the time evolves backward in time, the energy is nonincreasing and it follows that the backward heat equation is well-posed.
\begin{remark}
If we keep the domain to be $ \mathbb{R}^{+} \times U $ by imposing the boundary condition $ u(t, x) = 0 $ on $ \partial U \times [0, T) $, then we can get the same energy estimates as above. Here $ U \subset \mathbb{R}^{n} $ could be any bounded, open subset.
\end{remark} 
The other way to consider this backward heat equation is to set $ \tau = T - t \in [0, T] $ and with $ \tau $ we have
\begin{equation*}
\partial_{\tau} u - \Delta_{x} u = 0, u: [0, T] \times \mathbb{R}_{x}^{n} \rightarrow \mathbb{R}, u(0, x) = h(x).
\end{equation*}
\noindent Applying fundamental solution, we observe that
\begin{equation}
u(\tau, x) = \frac{1}{(4 \pi \tau)^{\frac{n}{2}}} \int_{\mathbb{R}^{n}} e^{\frac{\lvert x - y \rvert^{2}}{-4\tau}} h(y) dy \Rightarrow u(t, x) = \frac{1}{(4 \pi (T - t))^{\frac{n}{2}}}\int_{\mathbb{R}^{n}} e^{\frac{\lvert x - y \rvert^{2}}{4(t - T)}} h(y) dy.
\end{equation}
\noindent which implies that the solution of the backward heat equation will not blow up when backward in time starting at time $ T $. \\
\\ 
As inspired above by Perelman and Tao, we consider the embedding from $ [0, T] \times \mathbb{R}_{x}^{n} $ to $ [0, T] \times \mathbb{R}_{x}^{n} \times \mathbb{R}_{y}^{N} $ with the Riemannian metric $ \frac{2t}{N} dt^{2} + t d\omega_{1 \slash 2N}^{2} + dx^{2} $. \\
\\
As we discussed before, this is the same as applying $ t = \frac{\lvert y \rvert^{2}}{2N} = \frac{r^{2}}{2N} $ on Riemmanian metric $ dy^{2} + dx^{2} $. This fact can be seen as
\begin{equation*}
t = \frac{r^{2}}{2N} \Rightarrow dr^{2} = \frac{N}{2t} dt^{2} \Rightarrow dy^{2} + dx^{2} = dr^{2} + r^{2} d\omega_{1}^{2} + dx^{2} = \frac{2t}{N} dt^{2} + t d\omega_{\frac{1}{2N}}^{2} + dx^{2}.
\end{equation*}
\noindent This means we embed our original space into $ \mathbb{R}_{t}^{+} \times \mathbb{R}_{x}^{n} \times \mathbb{R}_{y}^{N} $ and consider the behavior of variable coefficient wave equation on the hyperserface $ t = \frac{\lvert y \rvert^{2}}{2N} $. \\
\\
This connection, as Perelman mentioned in \cite{Per1}, \cite{Per3}, \cite{Per2}, would switch the heat equation to a Laplace equation modulo $ N^{-1} $ in $ \mathbb{R}^{+} \times \tilde{M} $. Tao in \cite{Tao} loosely applied the change of variables $ t = \frac{\lvert y \rvert^{2}}{2N} $ and introduce the expression of this almost Laplace equation. We interpret the procedure here, and emphasize that this is a wave equation; we explain the advantage of this embedding, and trying to rigorously proof the connection between the solution of $ \tilde{\Delta} \tilde{u} = 0 $ modulo $ N^{-1} $ and the solution of heat equation as Perelman conjectured in his paper. This is our central contribution in this paper.
\begin{remark} We should point out that in Tao's notes \cite{Tao}, \cite{Tao9}, he interpreted the almost Laplace equation from the forward heat equation by using $ \tau = -\frac{ \lvert y \rvert^{2}}{2N} $, which I believe is a little bit problematic. Nevertheless, we greatly thank his contribution to this explicit expression, which Perelman and many others did not mention in their papers. We will go back to discuss this situation later.
\end{remark}

The equation $ t = \frac{\lvert y \rvert^{2}}{2N} $ provided us the relation between time evolution and spatial evolution, which can be given as below
\begin{equation*}
\frac{\partial}{\partial y_{i}} = \frac{\partial t}{\partial y_{i}} \frac{\partial}{\partial t} = \frac{y_{i}}{N} \frac{\partial}{\partial t} \Rightarrow \frac{\partial^{2}}{\partial y_{i}^{2}} = \frac{1}{N} \frac{\partial}{\partial t} + \frac{\partial t}{\partial y_{i}} \cdot \frac{y_{i}}{t} \frac{\partial^{2}}{\partial t^{2}} = \frac{1}{N} \frac{\partial}{\partial t} + \frac{y_{i}^{2}}{N^{2}} \frac{\partial^{2}}{\partial t^{2}}.
\end{equation*}
\noindent It follows that
\begin{equation}\label{fifteen}
\Delta_{y} = \sum_{i=1}^{N} \frac{\partial^{2}}{\partial y_{i}^{2}} = \frac{\partial}{\partial t} + \frac{\lvert y \rvert^{2}}{N^{2}} \frac{\partial^{2}}{\partial t^{2}} = \frac{\partial}{\partial t} + \frac{2t}{N} \frac{\partial^{2}}{\partial t^{2}} \Rightarrow \frac{\partial}{\partial t} = \Delta_{y} - \frac{2t}{N} \frac{\partial^{2}}{\partial t^{2}}.
\end{equation}
\noindent Apply (\ref{fifteen}) to (\ref{fourteen}), we get
\begin{equation}\label{sixteen}
(\frac{2t}{N} \cdot \frac{\partial^{2}}{\partial t^{2}} - \Delta_{x, y} ) u = 0, u(T, x, y) = h(x).
\end{equation}
\begin{remark} As we mentioned before, we achieved our first goal: we obtain a time reversible partial differential equation, i.e., $ u(-t, x, y) $ solves this equation if $ u(t, x, y) $ is a solution of this PDE, which allows us to begin our study of evolution of $ u $ at $ t = 0 $. We will explain more in Ricci flow related backward heat equation case.
\end{remark}
\begin{remark} If we consider further the space $ \mathbb{R}_{t > 0}^{+} \times \mathbb{R}_{x}^{n} \times \mathbb{R}_{y}^{N} $, then we observe that
\begin{equation}
P^{(N)} : = \frac{2t}{N}  \frac{\partial^{2}}{\partial t^{2}} - \Delta_{x, y}.
\end{equation}
\noindent is a strictly positive hyperbolic operator.
\end{remark}

We are now ready to provide a theorem which states the relationship between the solution of time-reversible equation (\ref{sixteen}) and solution of (\ref{fourteen}) backward in time by imposing some appropriate initial conditions. For simplicity, we assume that any boundary or initial condition is smooth enough and also integrable in $ \mathbb{R}_{x}^{n} $.
\begin{theorem} Given the Cauchy problem of variable coefficient wave equation
\begin{equation}\label{seventeen}
(\frac{2t}{N} \cdot \frac{\partial^{2}}{\partial t^{2}} - \Delta_{x, y} ) u = 0, (t, x, y) \in [t_{0}, T] \times \mathbb{R}_{x}^{n} \times \mathbb{R}_{y}^{N}, u(T, x, y) = h(x).
\end{equation}
\noindent provided that $ h(x) \in \mathcal{C}_{0}^{\infty}(\mathbb{R}^{n}) $. \\
Then for each $ N $, there exists a radial solution $ u = u^{(N)}(t, x, y) $ of this Cauchy problem with respect to $ y $ - coordinates, i.e. $ u^{(N)}(t, x, y) = u^{(N)}(t, x, Ay), \forall A \in SO(N) $, such that the limit $ u(t, x) : = \lim_{N \rightarrow \infty} u^{(N)}(t, x, y) |_{ \lbrace t = \lvert y \rvert^{2} \slash 2N \rbrace} $ exists at all points in $ [t_{0}, T] \times \mathbb{R}_{x}^{N} $. \\
Furthermore, the limit $ u(t, x) \in \mathcal{C}^{\infty}([t_{0}, T] \times \mathbb{R}_{x}^{n}) $ and solves the backward heat equation
\begin{equation}\label{eighteen}
\partial_{t} u + \Delta_{x} u = 0, u(T, x) = h(x), (t, x) \in [t_{0}, T] \times \mathbb{R}_{x}^{n}.
\end{equation}
\end{theorem}
\begin{remark} To prove this, we first need to show the existence of the solution of the Cauchy problem of the wave equation; we then need a notion of limit of wavefront sets $ WF(u^{(N)}) $ when $ N \rightarrow \infty $; we also need a result corresponds to the rotation-invariance of the Euclidean Laplacian $ \Delta_{y} $.
\end{remark}
\begin{remark} No matter in Euclidean case or Riemannian case, the existence of the solution of variable coefficient wave equation can only guarantee a weak solution in the distributional sense. In Theorem 4.1, we would like to have a result where the limit is obtained pointwisely. To this end, we would like the solution of (\ref{seventeen}) to have enough regularity, in particular, $ \mathcal{C}^{\infty} $ solution here, and this follows when the solution $ u $ of $ (\ref{seventeen}) $ satisfying $ u \notin \text{singsupp(u)} $, and specifically, it is outside $ WF(u) $.
\end{remark}
The following theorem states that variable coefficient wave equation $ Pu = f $ has a solution with some Sobolev regularity in Euclidean case, if we require the operator $ P $ to have following properties.
\begin{definition} {\bf \cite{Hor3}, Section 23.2.} We say that the operator $ P = \sum_{j = 0}^{m} P(t, x, D_{x}) D_{t}^{j} $ in $ \mathbb{R} \times \mathbb{R}^{n} $ is strictly hyperbolic and homogeneous if \\
(i) $ P_{m} = 1 $ and $ p_{j} \in S^{m - j} (\mathbb{R}^{n+1} \times \mathbb{R}^{n}) $ when $ j < m $; \\
(ii) $ P_{j} $ has a principal symbol $ p_{j}(t, x, \xi) $ which is homogeneous in $ \xi $ of degree $ m - j $; thus $ P_{j}(t, x, \xi) - (1 - \chi(\xi)) p_{j}(t, x, \xi) \in S^{m - 1 - j} $ if $ \chi \in \mathcal{C}_{c}^{\infty}(\mathbb{R}^{n}) $ is equal to $ 1 $ in a neighborhood of $ 0 $; \\
(iii) the zeros $ \tau $ of the principal symbol of $ P $
\begin{equation*}
p(t, x, \tau, \xi) = \sum_{0}^{m} p_{j}(t, x, \xi) \tau^{i}
\end{equation*}
\noindent are real when $ \xi \neq 0 $, and they are uniformly simple in the sense that for some positive constant $ c $
\begin{equation*}
\lvert \frac{\partial p(t, x, \tau, \xi)}{\partial \tau} \rvert \geqslant c \lvert \xi \rvert^{m - 1} \; if \; p(t, x, \tau, \xi) = 0.
\end{equation*}
\end{definition}
\begin{theorem} {\bf \cite{Hor3}, Theorem 23.2.2.} Assume that the hypotheses in Definition 4.1 are fulfilled. For arbitrary $ f \in L^{1}([0, T]; H^{s}) $ and $ \phi_{j} \in H^{s + m - 1 - j}, j < m $, there is then a unique solution $ u \in \bigcap_{0}^{m - 1} C^{j}([0, T]; H^{s + m - 1 - j}) $ of the Cauchy problem
\begin{equation*}
Pu = f \; if \; 0 < t < T; D_{t}^{j}u = \phi_{j} \; for \; j < m, \; if \; t = 0.
\end{equation*}
\end{theorem}
\begin{remark} The above theorem guarantees a global existence of the strong solution of (\ref{seventeen}) outside the wavefront set, which is good enough for us. Any standard PDE book provides the existence of the solution for second order variable coefficient hyperbolic PDEs, e.g. Evans book \cite{Evans}. There is another way to understand the operator $ P^{(N)} $. Within the space $ [t_{0}, T] \times \mathbb{R}_{x}^{n} \times (\mathbb{R}_{y}^{N} \backslash \lbrace 0 \rbrace) $, the operator could be observed as
\begin{equation*}
\tilde{P}^{(N)} = \frac{\partial^{2}}{\partial t^{2}} - \frac{N^{2}}{\lvert y \rvert^{2}} \Delta_{x, y},
\end{equation*}
\noindent by (\ref{fifteen}), and it is clear that $ P^{(N)} u = \tilde{P}^{(N)} u = 0 $. It follows from Chapter 18 of \cite{Hor3} that we could apply Hadamard construction to construct a paramatrix of $ \tilde{P}^{(N)} $ for the reason that $ - \frac{N^{2}}{\lvert y \rvert^{2}} \Delta_{x, y} $ is an elliptic operator. This paramatrix is a global version, and hence we can solve (\ref{seventeen}) by using this paramatrix.
\end{remark}
For each $ N $, the equation (\ref{seventeen}) has a set of singularity, in which the solution either does not exist or does not have real function representation, i.e., can only have distributional solution in the weak sense. The singularity of solutions of linear partial differential equations can be described by wavefront set. \\
\\
For a distribution $ u $ whose domain is some manifold $ M $, the wavefront set is defined to be
\begin{equation*}
WF(u) = \lbrace (x, \xi) \in T^{*}M : \xi \in \Sigma_{x}u \rbrace.
\end{equation*}
\noindent Here $ \Sigma_{x} u $, defined to be the singular fiber, is the complement of collection of elements $ \xi \in T_{x}^{*}M $ such that the Fourier transform of $ u $ localized at $ x $ has sufficient decay in a conical neighborhood of $ \xi $. \\
\\
Next theorem, named H\"ormander's propagation theorem, due to \cite{Hor3}, \cite{RBM}, explained what is the wavefront set looking like for any distribution $ u $, especially when $ u $ solves some partial differential equations $ Pu = f $.\\
\begin{theorem} {\bf \cite{RBM}, H\"ormander's Propagation Theorem.} Let $ M $ be a smooth manifold. If $ P \in \Psi^{m}(M) $ has real principal symbol and is properly supported then for any distribution $ u \in \mathcal{C}^{-\infty}(M) $,
\begin{equation*}
WF(u) \backslash WF(Pu) \subset \Sigma(P).
\end{equation*}
\noindent is a union of maximally null bicharacteristics in $ \Sigma(P) \backslash WF(Pu) $.
\end{theorem}
Since we would like to understand the asymptotic behavior of $ \lbrace u^{(N)} \rbrace $, it is natural to ask what will happen for the sequence of sets $ \lbrace WF(u^{(N)}) \rbrace $ when $ N $ approaches to infinity. \\
\\
In general, the wavefront set for $ u \in \mathcal{C}^{-\infty}(M) $ with some smooth manifold $ M $ is sitting in $ WF(u) \subset T^{*}M $. Since in our case we are discussing the asymptotic behavior of the sequence of sets $ \lbrace WF(u^{(N)}) \rbrace $ where each $ WF(u^{(N)}) \subset T^{*}([t_{0}, T] \times \mathbb{R}_{x}^{n} \times \mathbb{R}_{y}^{N}) $, it indicates that we should construct a limiting space which contains all $ \mathbb{R}_{y}^{N} $ when $ N \rightarrow \infty $ and thus we can discuss the limit of wavefront sets in the ambient space. \\
\\
Note that the Euclidean metric in each $ \mathbb{R}_{y}^{N} $ is the same as the $ l^{2} $ - norm with the canonical identification that $ \mathbb{R}_{y}^{N} \cong T_{y_{0}}\mathbb{R}_{y}^{N} $ at each point. We now set the limiting space, which can be denoted as $ \mathbb{R}_{y}^{\infty} : = \bigoplus_{i=1}^{\infty} \mathbb{R} $, the direct sum of countably copies of $ \mathbb{R} $. It is exactly the set of collection of countable sequence $ (y_{1}, \dotso, y_{n}, \dotso) $ such that all but finitely many elements are zero. $ \mathbb{R}_{y}^{\infty} $ is a pre-Hilbert space with the $ l^{2} $ - norm also. Define
\begin{align*}
\phi_{N} & : \mathbb{R}_{y}^{N} \rightarrow \mathbb{R}_{y}^{\infty}, (y_{1}, \dotso, y_{N}) \mapsto (y_{1}, \dotso, y_{N}, 0, 0, \dotso); \\
\pi_{N} & : \mathbb{R}_{y}^{\infty} \rightarrow \mathbb{R}_{y}^{N} , (y_{1}, \dotso, y_{N}, y_{N+1}, \dotso) \mapsto (y_{1}, \dotso, y_{N}).
\end{align*}
\noindent From the discussion above, we can see that each $ \phi_{N} $ is an isometric embedding from $ \mathbb{R}_{y}^{N} $ to $ \mathbb{R}_{y}^{\infty} $, and it follows that
\begin{equation*}
Id_{t} \times Id_{x} \times \phi_{N} : \mathbb{R}_{t} \times \mathbb{R}_{x}^{n} \times \mathbb{R}_{y}^{N} \rightarrow \mathbb{R}_{t} \times \mathbb{R}_{x}^{n} \times \mathbb{R}_{y}^{\infty},
\end{equation*}
\noindent is an isometric embedding for each $ N $. This setup allows us to discuss the closeness between two points in two different spaces $ \mathbb{R}_{t} \times \mathbb{R}_{x}^{n} \times \mathbb{R}_{y}^{N_{1}} $ and $ \mathbb{R}_{t} \times \mathbb{R}_{x}^{n} \times \mathbb{R}_{y}^{N_{2}} $: we embed points in two different spaces to the same limiting space $ \mathbb{R}_{t} \times \mathbb{R}_{x}^{n} \times \mathbb{R}_{y}^{\infty} $ with the distance function induced by the products of standard norm in time axis, in $ \mathbb{R}_{x}^{n} $ and the $ l^{2} $ - norm in $ \mathbb{R}_{y}^{\infty} $. \\
\\
In our case, $ (t, x, y, \tau, \xi, \eta) \in WF(u^{(N)}) \subset T^{*}(\mathbb{R}_{t} \times \mathbb{R}_{x}^{n} \times \mathbb{R}_{y}^{N}) \cong (\mathbb{R}_{t} \times \mathbb{R}_{x}^{n} \times \mathbb{R}_{y}^{N})^{2} $. In this setup, we write down the distance function to be
\begin{align*}
\lvert (t_{1}, x_{1}, y_{1}, \tau_{1}, \xi_{1}, \eta_{1}) - (t_{1}, x_{1}, y_{1}, \tau_{1}, \xi_{1}, \eta_{1}) \rvert_{WF_{\infty}} & = \lvert t_{1} - t_{2} \rvert_{\mathbb{R}_{t}}  + \lvert x_{1} - x_{2} \rvert_{\mathbb{R}_{x}^{n}} + \lvert y_{1} - y_{2} \rvert_{l_{\infty}^{2}} + \\
& + \lvert \tau_{1} - \tau_{2} \rvert_{\mathbb{R}_{t}}  + \lvert \xi_{1} - \xi_{2} \rvert_{\mathbb{R}_{x}^{n}} + \lvert \eta_{1} - \eta_{2} \rvert_{l_{\infty}^{2}}.
\end{align*}
\noindent Here $ \lvert \cdot \rvert_{l_{\infty}^{2}} $ is the norm on $ \mathbb{R}_{y}^{\infty} = \bigoplus_{i=1}^{\infty} \mathbb{R} $. \\
\\
With the canonical isomorphism between Euclidean space and its cotangent space, we can apply the same limiting space on cotangent space of $ \mathbb{R}_{y}^{N} $ and define as follows:
\begin{definition}
Given the sequence of spaces $ \mathbb{R}_{t} \times \mathbb{R}_{x}^{n} \times \mathbb{R}_{y}^{N} $, we define the limit of the sequence of wavefront sets $ WF(u^{(N)}) : = WF_{N} $, where $ u^{(N)} $ is a distribution on $ \mathbb{R}_{t} \times \mathbb{R}_{x}^{n} \times \mathbb{R}_{y}^{N} $, to be
\begin{equation}\label{nineteen}
WF_{\infty} : = \lim_{N \rightarrow \infty} WF_{N}, WF_{\infty} \subset (\mathbb{R}_{t} \times \mathbb{R}_{x}^{2} \times \mathbb{R}_{y}^{\infty})^{2}.
\end{equation}
\noindent We say $ (t, x, y, \tau, \xi, \eta) \in WF_{\infty} $ if: \\
(i)for every $ \epsilon > 0 $, there exists a $ N_{0} \in \mathbb{N} $ such that
\begin{equation*}
\lvert (t, x, y, \tau, \xi, \eta) - (t^{(N)}, x^{(N)}, y^{(N)}, \tau^{(N)}, x^{(N)}, \eta^{(N)}) \rvert_{WF_{\infty}} < \epsilon, \forall N \geqslant N_{0}.
\end{equation*}
\noindent or (ii) for every positive number $ M > 0 $, there exists $ N_{0} \in \mathbb{N} $ such that
\begin{equation*}
\lvert t^{(N)}, x^{(N)}, y^{(N)}, \tau^{(N)}, x^{(N)}, \eta^{(N)} \rvert_{WF_{\infty}} > M, \forall N \geqslant N_{0}.
\end{equation*}
\noindent In case (ii), the limit $ (t, x, y, \tau, \xi, \eta) $ has some or all entries to be infinity. \\
\par
Here $ (t, x, y, \tau, \xi, \eta) \in (\mathbb{R}_{t} \times \mathbb{R}_{x}^{n} \times \mathbb{R}_{y}^{\infty})^{2} $, $ (t^{(N)}, x^{(N)}, y^{(N)}, \tau^{(N)}, x^{(N)}, \eta^{(N)}) \in (\mathbb{R}_{t} \times \mathbb{R}_{x}^{n} \times \mathbb{R}_{y}^{N})^{2} \hookrightarrow (\mathbb{R}_{t} \times \mathbb{R}_{x}^{n} \times \mathbb{R}_{y}^{\infty})^{2} $, and the distance function $ \lvert \cdot \rvert $ on $ (\mathbb{R}_{t} \times \mathbb{R}_{x}^{n} \times \mathbb{R}_{y}^{\infty})^{2} $ is induced by the norm discussed above.
\end{definition}
\begin{remark}
The definition of $ WF_{\infty} $ indicates a concept of essential singularities for the sequence of distributions, telling us which points, especially the pair $ (t, x) $ in the original space, cannot be maneuvered by embedding into a larger space. Equivalently, it says that all singularities resulted from geometry of the space in $ y $ are removed, when we send $ N \rightarrow \infty $.
\end{remark}
\begin{remark}
The above discussion and Definition 4.2 also works when we apply them to the space $ [t_{0}, T] \times \mathbb{R}_{x}^{n} \times \mathbb{R}_{y}^{N} $ and the cotangent spaces for each $ N $ are the same, so is for the limiting space.
\end{remark}
To show the theorem, we need the proposition below, which states the rotational invariance of Euclidean Laplacian.
\begin{proposition} On the Euclidean space $ \mathbb{R}^{n} $ the Euclidean Laplacian $ \Delta_{x} = \sum_{i=1}^{n} \frac{\partial^{2}}{\partial x_{i}^{2}} $ is rotationally invariant in the sense that for $ u : \mathbb{R}^{n} \rightarrow \mathbb{R}, v(x) = u(Ax), \forall A \in SO(n) $, we have
\begin{equation*}
\Delta_{x} v(x) = \Delta_{y} u(y), \text{where} \; y = Ax.
\end{equation*}
\end{proposition}
\begin{proof} Denote $ A = (A_{ij}) $, we compute that
\begin{equation*}
\frac{\partial v(x)}{\partial x_{i}} = \frac{\partial u(y)}{\partial x_{i}} = \sum_{j=1}^{n} \frac{\partial u(y)}{\partial y_{j}} \frac{\partial y_{j}}{\partial x_{i}} = \sum_{j = 1}^{n} \frac{\partial u(y)}{\partial y_{j}} A_{ji}.
\end{equation*}
\noindent Taking one more derivative we observe that
\begin{equation*}
\frac{\partial^{2} v(x)}{\partial x_{i}^{2}} = \frac{\partial}{\partial x_{i}} (\sum_{j = 1}^{n} \frac{\partial u(y)}{\partial y_{j}} A_{ji}) = \sum_{k = 1}^{n} \sum_{j = 1}^{n} \frac{\partial^{2} u(y)}{\partial y_{j} y_{k}} A_{ji} A_{ki}.
\end{equation*}
\noindent Since $ A \in SO(n) $, hence we have
\begin{equation*}
I_{n} = AA^{T} \Rightarrow \sum_{i=1}^{n} A_{ji} A_{ki} = (AA^{T})_{kj} = \delta_{kj}.
\end{equation*}
\noindent Hence it follows that
\begin{equation*}
\Delta_{x} v(x) = \sum_{i=1}^{n} \sum_{k = 1}^{n} \sum_{j = 1}^{n} \frac{\partial^{2} u(y)}{\partial y_{j} y_{k}} A_{ji} A_{ki} = \sum_{j = 1}^{n} \frac{\partial^{2} u(y)}{\partial y_{j}^{2}} = \Delta_{y} u(y).
\end{equation*}
\end{proof}
Now we can proof our main theorem.

\begin{proof} {\bf Proof of Theorem 4.1.} To conclude (\ref{seventeen}) has a solution in the weak sense, we check that (\ref{seventeen}) satisfies all hypothesis in Theorem 4.2. Since $ t \in [t_{0}, T] \subset (0, \infty) $ and $ t_{0} > 0 $ strictly, we can rewrite (\ref{seventeen}) as
\begin{equation}\label{one-and-one}
\frac{2t}{N} \cdot \frac{\partial^{2}}{\partial t^{2}} u - \Delta_{x, y}u = 0 \Leftrightarrow \frac{\partial^{2}}{\partial t^{2}} u - \frac{N}{2t} \Delta_{x, y} u = 0 \Leftrightarrow D_{t}^{2} u - \frac{N}{2t} (\sum_{i=1}^{n} D_{x_{i}}^{2} + \sum_{j=1}^{N} D_{y_{j}}^{2}) u : = \tilde{P}^{(N)} u = 0.
\end{equation} 
\noindent It follows clearly that if we write $ \tilde{P}^{(N)} = \sum_{j=0}^{2} P_{j}(t, x, y, D_{x, y})D_{t}^{j} $, we have $ P^{2} = 1 $, $ P_{1} = 0 $ and $ P_{0} \in S^{2}(\mathbb{R}^{n + N + 1} \times \mathbb{R}^{n + N}) $, so (i) is satisfied. \\
\\
Consequently, we observe that $ p_{0}(t, x, y, \xi, \eta) = \frac{t}{2N} (\lvert \xi \rvert^{2} + \lvert \eta \rvert^{2}) $ and clearly (ii) is hold. \\
\\
Lastly, we observe that the principal symbol of $ \tilde{P}^{(N)} $, denoted by $ \tilde{p}^{(N)} $, is of the form
\begin{equation*}
\tilde{p}^{(N)} = \tau^{2} - \frac{t}{2N}(\lvert \xi \rvert^{2} + \lvert \eta \rvert^{2}).
\end{equation*}
\noindent Clearly it has real solutions when $ \lvert \xi \rvert + \lvert \eta \rvert \neq 0 $ and it is uniformly simple, since we have
\begin{equation*}
\frac{\partial \tilde{p}^{(N)}}{\partial \tau} = 2 \tau = 2 \cdot \sqrt{\frac{t}{2N}(\lvert \xi \rvert^{2} + \lvert \eta \rvert^{2})}.
\end{equation*}
\noindent And thus the inequality in (iii) of Definiition 4.1 is satisfied, since $ t \geqslant t_{0} > 0 $. \\
\\
Hence we conclude that, by (\ref{one-and-one}), the wave equation (\ref{seventeen}) has a weak solution in the sense of Theorem 4.2. Furthermore, we conclude by Theorem 4.3 that (\ref{seventeen}) has a smooth solution outside $ WF(u) $. \\
\\
We first show that for each $ N $, the limit does exist and the projection of the limit onto $ \mathbb{R}_{t} \times \mathbb{R}_{x}^{n} $ solves the backward heat equation (\ref{eighteen}) strongly. To do this, we need to construct a strong solution of (\ref{seventeen}) that is rotationally invariant in $ y $, say $ u^{(N)}(t, x, y) = u^{(N)}(t, x, \lvert y \rvert) $. Thus alongside each $ \mathbb{S}^{N} $ fibers with the relation $ t = \frac{\lvert y \rvert^{2}}{2N} $, $ u^{(N)}(t, x, \lvert y \rvert) = u^{(N)}(t, x, \sqrt{2Nt}) $ which allows us to consider the limit on the space $ [t_{0}, T] \times \mathbb{R}_{x}^{n} $. \\
\\
Given $ \tilde{u}^{(N)}(t, x, y) $ that solves (\ref{seventeen}), we consider $ \tilde{u}^{(N)}(t, x, Ay), A \in SO(N) $. Observe that
\begin{align*}
& \text{ $ \tilde{u}^{(N)}(t, x, y) $ solves (\ref{seventeen})} \Rightarrow \frac{2t}{N} \partial_{tt} \tilde{u}^{(N)}(t, x, y) - \Delta_{x}\tilde{u}^{(N)}(t, x, y) - \Delta_{y}\tilde{u}^{(N)}(t, x, y) = 0 \\
\Rightarrow & \frac{2t}{N} \partial_{tt} \tilde{u}^{(N)}(t, x, y) - \Delta_{x}\tilde{u}^{(N)}(t, x, y) - \Delta_{y}\tilde{u}^{(N)}(t, x, y) = 0, \text{$ y $ is a dummy variable} \\
\Rightarrow & \frac{2t}{N} \partial_{tt} \tilde{u}^{(N)}(t, x, y) - \Delta_{x}\tilde{u}^{(N)}(t, x, y) - \Delta_{y}\tilde{u}^{(N)}(t, x, y) = 0, \text{by Proposition 4.1}.
\end{align*}
\noindent Hence we conclude that $ \tilde{u}^{(N)}(t, x, Ay) $ solves (\ref{seventeen}), for all $ A \in SO(N) $, by Proposition 4.1. \\
\\
Since $ SO(N) $ is a compact Riemannian manifold, we can define
\begin{equation}\label{twenty}
u^{(N)}(t, x, y) =\frac{1}{\text{Vol}(SO(N))} \int_{SO(N)} \tilde{u}^{(N)}(u, x, Ay) d\mu_{N},
\end{equation}
\noindent where $ d\mu_{N} $ is the volume form on $ SO(N) $. \\
\par
It follows immediately that $ u^{(N)}(t, x, y) $ defined above is rotationally invariant, since it is the average over $ SO(N) $. Since $ SO(N) $ is a compact manifold, Dominated Convergence Theorem allows us to pass derivatives over the integration and hence we conclude that $ u^{(N)}(t, x, y) $ defined in (\ref{twenty}) solves (\ref{seventeen}). In particular, (\ref{twenty}) is constant along the hypersurface $ t = \lvert y \rvert^{2} \slash 2N $. \\
\par
We now discuss the wavefront set of $ u^{(N)} $ defined in (\ref{nineteen}). The issue for this definition is that we would like to make sure that $ u^{(N)} $ is well-defined outside the wavefront set of $ \tilde{u}^{(N)} $, which means that the rotational action will not arise new singularities. We examine both $ WF(\tilde{u}^{(N)}) $ and $ WF(u^{(N)}) $. \\
\\
According to the application of Theorem 3.3, we first should check that the principal symbol
\begin{equation*}
p^{(N)} = \frac{2t}{N} \tau^{2} - \lvert \xi \rvert^{2} - \lvert \eta \rvert^{2},
\end{equation*}
\noindent of our hyperbolic operator $ P^{(N)} $ is of real type and properly supported. $ P^{(N)} $ is definitely properly supported, since $ \text{supp}(P^{(N)} u) \subset \text{supp}(u) $ for all $ u $ acting on $ [t_{0}, T] \times \mathbb{R}_{x}^{n} \times \mathbb{R}_{y}^{N} $. This follows from two facts, which could be found in (cf. Melrose, 1990) that on smooth manifold $ M $, the distribution $ u \in \mathcal{C}^{-\infty}(M) $ satisfies
\begin{equation*}
\text{supp}(D^{\alpha} u) \subset \text{supp}(u), \text{supp}(\phi u) \subset \text{supp u}, \forall \phi \in \mathcal{C}^{\infty}(M).
\end{equation*}
\noindent Here $ D^{\alpha} = (-i)^{\lvert \alpha \rvert} \partial^{\alpha} $ is the differential operator, and $ \phi u $ can be considered as the group action $ \mathcal{C}^{\infty}(M) \times \mathcal{C}^{-\infty}(M) \rightarrow \mathcal{C}^{-\infty}(M) $. \\
\\
Then by definition of properly supported pseudodifferential operators, we observe that $ \text{supp}(P^{(N)} u) \subset K $ if $ \text{supp}(u) \subset K $ for any compact set $ K $ provided that $ u \in \mathcal{C}^{-\infty}([t_{0}, T] \times \mathbb{R}_{x}^{n} \times \mathbb{R}_{y}^{N}) $. \\
\\
To verify that $ P^{(N)} $ is of real principal type, we need to check that $ p^{(N)} $ is real and $ dp^{(N)} \neq 0 $ away from the point $ \lvert \xi \rvert = \lvert \eta \rvert = \tau = 0 $. It is clear from the expression $ p^{(N)} = \frac{2t}{N} \tau^{2} - \lvert \xi \rvert^{2} - \lvert \eta \rvert^{2} $ that $ p^{(N)} $ is real, and we can compute that
\begin{equation*}
dp^{(N)} = \frac{2}{N} \cdot \tau^{2} dt + \frac{2t}{N} \cdot 2 \tau d\tau - \sum_{i=1}^{n} 2\xi_{i} d\xi_{i} - \sum_{j=1}^{N} 2\eta_{j} d\eta_{j}.
\end{equation*}
\noindent Clearly $ dp^{(N)} \neq 0 $ when it is away from the point $ \lvert \xi \rvert = \lvert \eta \rvert = \tau = 0 $. It follows that $ P^{(N)} $ has real principal symbol. \\
\\
Combined with both verifications, we can apply Theorem 4.3 to $ \tilde{u}^{(N)} $. It then follows that
\begin{equation*}
WF(\tilde{u}^{(N)}) \subset WF(P^{(N)} \tilde{u}^{(N)} ) \cup \Sigma(P^{(N)}) = \Sigma(P^{(N)}),
\end{equation*}
\noindent since $ \tilde{u}^{(N)} $ is a solution of $ P^{(N)} \tilde{u}^{(N)} = 0 $ hence it only depends on the characteristic set of $ P^{(N)} $ and hence is solely determined by the principal symbol $ p^{(N)} $, the union of maximally null bicharacteristics. \\
\\
We would like to show that
\begin{equation}
WF(u^{(N)}) = WF(\tilde{u}^{(N)}) \subset \Sigma(P^{(N)}).
\end{equation}
\noindent To show this, we need to compute the characteristic set of $ \Sigma(P^{(N)}) $, since the propagation of singularities is alongside the integral curves of Hamiltonian vector field of $ p^{(N)} $ within the characteristic. We would like to show that the integral curves are radial also and thus we can conclude (34). \\
\\
Given the principal symbol $ p^{(N)} = \frac{2t}{N} \tau^{2} - \lvert \xi \rvert^{2} - \lvert \eta \rvert^{2} $ for fixed $ N $, it's associated Hamiltonian vector field $ H_{p^{(N)}} $ is given by
\begin{equation*}
H_{p^{(N)}} = \frac{2t}{N} \cdot 2 \tau \frac{\partial}{\partial t} - \sum_{i=1}^{n} 2\xi_{i} \frac{\partial}{\partial x_{i}} - \sum_{j=1}^{N} 2\eta_{j} \frac{\partial}{\partial y_{j}} - \frac{2\tau^{2}}{N} \frac{\partial}{\partial \tau}.
\end{equation*}
\noindent Within $ \Sigma(P^{(N)}) = (p^{(N)})^{-1}(0) = \lbrace \frac{2t}{N} \tau^{2} - \lvert \xi \rvert^{2} - \lvert \eta \rvert^{2} = 0 \rbrace $, we can compute the integral curves $ (t(s), x(s), y(s), \tau(s), \xi(s), \eta(s)) $ with respect to $ H_{p^{(N)}} $:
\begin{align*}
& \dot{t} = \frac{2t}{N} \cdot 2 \tau, \dot{x}_{i} = -2\xi_{i}, \dot{y}_{j} = - 2\eta_{j}, \\
& \dot{\tau} = -\frac{2\tau^{2}}{N}, \dot{\xi}_{i} = 0, \dot{\eta}_{j} = 0.
\end{align*}
\noindent Solving this linear ODE system with initial conditions $ (t(0), x(0), y(0), \tau(0), \xi(0), \eta(0)) $, we have
\begin{align}\label{twentyone}
& \tau(s) = \frac{\tau(0)N}{2s\tau(0) + N}, \xi_{i}(s) = \xi_{i}(0), \eta_{j}(s) = \eta_{j}(0), \nonumber \\
& x_{i}(s) = x_{i}(0) - 2\xi_{i}(0) \cdot s, y_{j}(s) = y_{j}(0) - 2\eta_{j}(0) \cdot s, \nonumber \\
& t(s) = 4\tau(0)^{2} t(0) (\frac{s}{N} + \frac{1}{2\tau(0)})^{2},
\end{align}
\noindent provided that $ \tau(0) \neq 0 $. When $ \tau(0) = 0 $, we then have the trivial solution
\begin{align}\label{twentytwo}
& t(s) = t(0), x_{i}(s) = x_{i}(0) - 2\xi_{i}(0)s, y_{j}(s) = y_{j}(0) - 2\eta_{j}(0) s, \nonumber \\
& \tau(s) = \tau(0) = 0, \xi_{i}(s) = \xi_{i}(0), \eta_{j}(s) = \eta_{j}(0).
\end{align}
\noindent From (\ref{twentyone}), we can obtain an equation corresponds to $ t(s), x(s), y(s) $, which is
\begin{equation}\label{twentythree}
t(s) = 4\tau(0)^{2} t(0) (\sqrt{\frac{\lvert x(s) - x(0) \rvert^{2} + \lvert y(s) - y(0) \rvert^{2}}{8\tau(0)^{2} t(0) N}} + \frac{1}{2 \tau(0)})^{2}.
\end{equation}
\noindent (\ref{twentythree}) holds since we are considering the null characteristic, i.e. we have $ p^{(N)} = 0 $ in this case. \\
\\
For the integral curves (\ref{twentyone}), (\ref{twentytwo}) within $ \Sigma(P^{(N)}) $, the extra equation it should obey is $ \frac{2t}{N} \tau^{2} - \lvert \xi \rvert^{2} - \lvert \eta \rvert^{2} = 0 $, and it definitely holds when $ s = 0 $, we can see that
\begin{equation*}
\frac{2t(0)}{2N} \tau(0)^{2} - \lvert \xi(0) \rvert^{2} - \lvert \eta(0) \rvert^{2} = 0 \Rightarrow \frac{2t(0)}{N} (\tau(0))^{2} - \lvert \xi(0) \rvert^{2} - \lvert A\eta(0) \rvert^{2} = 0, \forall A \in SO(N).
\end{equation*}
\noindent It follows that if $ \eta(0) \in WF(\tilde{u}^{(N)}) $ then so is $ A\eta(0) $. From (\ref{twentyone}), (\ref{twentytwo}) and (\ref{twentythree}), we observe that $ Ay(s), A\eta(s) \in WF(\tilde{u}^{(N)}) $ if and only if $ Ay(0), A\eta(0) \in WF(\tilde{u}^{(N)}) $. Therefore, it follows that $ WF(\tilde{u}^{(N)}) $ contains all rotation alongside $ y $ - direction and it follows from definition (\ref{nineteen}) that (\ref{twenty}) holds. Furthermore, the propagation of singularities of $ u^{(N)} $ which solves (\ref{seventeen}) are along integral curves (\ref{twentyone}) or (\ref{twentytwo}). \\
\\
Specifically, this indicates that the definition (\ref{nineteen}) is well-defined and the average over all rotational solutions of (\ref{seventeen}) did not expand the wavefront set, hence we conclude (34). \\
\\
We now have a solution of (\ref{seventeen}) radial in $ y $ - coordinates, i.e., $ u^{(N)}(t, x, y) = u^{(N)}(t, x, \lvert y \rvert) $. Alongside the hyperplane $ t = \frac{\lvert y \rvert^{2}}{2N} $, we conclude that $ u^{(N)} $ is constant alongside each $ \mathbb{S}^{N} $ fiber whose curvature is $ \frac{1}{2N} $. We would like to show that outside $ WF(u^{(N)}) $, the projection of $ u^{(N)} $ onto $ (t, x) $ - plane solves (\ref{eighteen}). \\
\\
Since $ u^{(N)}(t, x, y) = u^{(N)}(t. x, \sqrt{2Nt}) $ is parametrized by $ t, x $ on $ \lbrace t = \frac{\lvert y \rvert^{2}}{2N} \rbrace $, we have
\begin{align*}
0 & = (\frac{2t}{N} \cdot \frac{\partial^{2}}{\partial t^{2}} - \Delta_{x} ) u^{(N)}(t, x, \lvert y \rvert) - \sum_{j=1}^{N} \frac{\partial^{2}u^{(N)}(t, x, \lvert y \rvert)}{\partial y_{j}^{2}} \\
& = (\frac{\sum_{i=1}^{N} y_{i}^{2}}{N^{2}} \cdot \frac{\partial^{2}}{\partial t^{2}} - \Delta_{x} ) u^{(N)}(t, x, \sqrt{2Nt}) - \sum_{j=1}^{N} (\frac{1}{N} \frac{\partial }{\partial t} + \frac{y_{j}^{2}}{N^{2}} \frac{\partial^{2} }{\partial t^{2}})u^{(N)}(t, x, \sqrt{2Nt}) \\
& = - \Delta_{x} u^{(N)}(t, x, \sqrt{2Nt}) - \frac{\partial u^{(N)}(t, x, \sqrt{2Nt})}{\partial t}.
\end{align*}
\noindent It follows that $ u^{(N)}|_{\lbrace t = \lvert y \rvert^{2} \slash 2N \rbrace} $ solves (\ref{eighteen}) outside $ WF(u^{(N)}) $ for each $ N \in \mathbb{N} $. \\
\\
If a point $ (t, x, y, \tau, \xi, \eta) \notin WF(u^{(N)}) $ for all large $ N $, then we conclude that at this point, $ u(t, x) = \lim_{N \rightarrow \infty} u^{(N)}(t, x, y) $ does exist in the strong sense and solves the backward heat equation by above argument. It is the same as saying that $ (t, x, y, \tau, \xi, \eta) \notin WF_{\infty} $. \\
\\
To show that the limit does exist at all points, we would like to show that $ WF_{\infty} : = \lim_{N \rightarrow \infty} WF(u^{(N)}) $ defined in Definition 4.1, contains only the initial points, i.e. collection of points $ (t_{0}, x, y, \tau(0), \xi(0), \eta(0)) $ within the null bicharacteristics, as well as $ (t, \infty, \infty, \tau(0), \xi(0), \eta(0)) $, whose projection onto the $ \mathbb{R}_{x}^{n} $ is at infinity, the extreme faraway.
\begin{remark} We fix $ N $ in above discussions including the derivation of null bicharacteristics, the existence of limit, and solvability of (\ref{seventeen}) and (\ref{eighteen}) so we use the somehow vague notation $ (t, x, y, \tau, \xi, \eta) $. For discussion below, we consider the collection of all $ u^{(N)} $ and their corresponding wavefront sets, hence we apply same notation in Definition 4.1 to denote $ (t^{(N)}, x^{(N)}, y^{(N)}, \tau^{(N)}, \xi^{(N)}, \eta^{(N)}) $ as points in $ T^{*}([t_{0}, T] \times \mathbb{R}_{x}^{n} \times \mathbb{R}_{y}^{N}) $.
\end{remark}
We begin with the case $ \tau^{(N)}(0) = 0 $ for each $ N $, in this case we apply (\ref{twentytwo}) and the fact that $ p^{(N)}(t^{(N)}, x^{(N)}, y^{(N)}, \tau^{(N)}, \xi^{(N)}, \eta^{(N)}) = 0 $, it follows that
\begin{align*}
& \frac{2t^{(N)}(0)}{N} \tau^{(N)}(0)^{2} - \lvert \xi^{(N)}(0) \rvert^{2} - \lvert \eta^{(N)}(0) \rvert^{2} = 0 \\
\Rightarrow & \lvert \xi^{(N)}(0) \rvert^{2} + \lvert \eta^{(N)}(0) \rvert^{2} = 0 \Rightarrow \xi^{(N)} = \eta^{(N)} = 0, \forall N \\
\Rightarrow & x^{(N)}(s) = x^{(N)}(0), y^{(N)}(s) = y^{(N)}(0), \xi^{(N)}(s) = \eta^{(N)}(s) = 0. \; \text{Using (\ref{twentytwo})}.
\end{align*}
\noindent Apply the distance function $ \lvert \cdot \rvert_{WF_{\infty}} $, we conclude that
\begin{equation*}
(t(0), x(0), y(0), 0, 0, 0) \in WF_{\infty}.
\end{equation*}
\noindent where $ t(0) = t^{(N)}(0) = T, x(0) = x^{(N)}(0), \forall N $ and $ y(0) = (y_{1}^{(N)}(0), \dotso, y_{N}^{(N)}(0), 0, 0, \dotso) $. Note that $ (t(0), x(0), y(0)) $ indicates a collection of points, since we can choose different initial points for null bicharacteristics. It follows that all initial points of the backward heat equation are contained in $ WF_{\infty} $. \\
\\
The other case corresponds to $ \tau^{(N)}(0) \neq 0 $. \\
Apply the fact that $ p^{(N)}(t^{(N)}, x^{(N)}, y^{(N)}, \tau^{(N)}, \xi^{(N)}, \eta^{(N)}) = 0 $, we see that for given $ \epsilon > 0 $, when $ N $ large enough, we must have
\begin{equation*}
\lvert \xi^{(N)}(0) \rvert^{2} + \lvert \eta^{(N)}(0) \rvert^{2} < \epsilon \Rightarrow \lvert \xi^{(N)}(s) \rvert < \epsilon, \lvert \xi^{(N)}(s) \rvert < \epsilon.
\end{equation*}
\noindent when $ N $ is large enough. This implies that at $ s = 0 $, then when $ N $ large enough, we have
\begin{equation}\label{twentyfour}
\lvert (t(0), x(0), y(0), \tau(0), 0, 0) - (t^{(N)}, x^{(N)}, y^{(N)}, \tau^{(N)}, x^{(N)}, \eta^{(N)}) \rvert_{WF_{\infty}} < \epsilon.
\end{equation}
\noindent Here $ t(0), x(0), y(0) $ have the same interpretation as above and $ \tau(0) = \tau^{(N)}(0) $. \\
According to (\ref{twentyone}), we can obtain $ t^{(N)}(s) (\tau^{(N)}(s))^{2} = t^{(N)}(0) (\tau^{(N)}(0))^{2} $ and since $ \tau^{(N)}(0) \neq 0 $, it follows that $ t(s) \neq t(0) $ when $ s \neq 0 $, then from (\ref{twentythree}), we observe that when we arrive at some $ t \in [t_{0}, T] $ after some time $ s $, the integral curves travel to points $ (x^{(N)}(s), y^{(N)}(s)) $ in a way that
\begin{equation*}
\lvert x^{(N)}(s) - x^{(N)}(0) \rvert^{2} = O(N), \lvert y^{(N)}(s) - y^{(N)}(0) \rvert^{2} = O(N).
\end{equation*}
\noindent It follows that when $ N $ increases, the singularities we achieve are far and far away and eventually we conclude by Definition 4.1 that 
\begin{equation}\label{twentyfive}
(t(s), \infty, \infty, \tau(s), \xi(s), \eta(s)) \in WF_{\infty}.
\end{equation}
\noindent It indicates that essentially, $ (t, \infty) \in [t_{0}, T] \times \mathbb{R}_{x}^{n} $ amounts to the only singularity of (\ref{seventeen}) as well as (\ref{eighteen}) in the limiting sense. It follows that $ u(t, x) = \lim_{N} u^{(N)}(t, x, y) |_{\lbrace t = \frac{\lvert y \rvert^{2}}{2N} \rbrace} $ does exist at all points $ [t_{0}, T) \times \mathbb{R}_{x}^{n} $ and solves (\ref{eighteen}). \\
\\
The only left points $ (t(0), x(0), y(0), \tau(0), 0, 0) $ in $ WF_{\infty} $ shows that we may not have an explicit expression for either $ u(t, x) $ or $ u^{(N)}(t, x, y) $ at the initial point $ t = T $, but the initial conditions $ H, h, \tilde{H} $ are satisfied by limiting argument. For example, we can easily deduce the fact that $ u(t, x) \rightarrow H(x), t \rightarrow t_{0} $. \\
\\
Lastly, all initial conditions are satisfied due to the construction of $ H, \tilde{H} $ and the uniqueness of the solution of (\ref{eighteen}).
\end{proof}
\begin{remark} The equation (\ref{seventeen}) can be rewrite formally as
\begin{equation*}
\frac{\partial^{2} u^{(N)}}{\partial t^{2}} = \frac{N}{2t} \Delta_{x, y} u^{(N)}.
\end{equation*}
\noindent The speed of propagation for each $ u^{(N)} $ is now $ \sqrt{N \slash 2t} $, which increases as $ N $ increases, and heuristically, we observe that the property of infinite speed of propagation, which describes by heat equation, could be approximated by increasingly finite speed of propagation. \\
\\
On the other hand, if we agree that the light speed $ c $ is the upper limit, then the above equation indicates that although we may have nearly infinite speed, we should wait a while to observe the behavior of $ u $ or $ u^{(N)} $ when we are in the space with large dimension.
\end{remark}
\begin{remark} (Suspicious remark) (\ref{twentyfour}) says that at the initial points, the singularities will only move alongside the $ \tau $ - direction or equivalently, $ \partial_{t} u $ may not exist, which corresponds to the heat kernel. (\ref{twentyfive}) says that at infinity, the singularities could arrive alongside each direction, it follows that when $ t = T $, then the solution may loss regularities alongside all coordinates.
\end{remark}
For the backward heat equation (\ref{eighteen}), if we assume the initial condition $ u(T, x) = h(x) \in \mathcal{C}_{0}^{\infty}(\mathbb{R}) $ at arbitrarily large $ T $, we can easily extend the above theorem to the interval $ [t_{0}, T] $ for arbitrarily large $ T $. For example, in the Ricci flow case, $ \partial_{t} u = - R $, the scalar curvature. Thus if the Ricci flow can flow forward for all time then definitely we can consider the corresponding backward equation (17), or gradient steady/shrinking/expanding solitons on the time interval $ [t_{0}, + \infty) $. One particular example for this is the $ \kappa $ - solution in three dimensional manifolds. \\

\noindent Instead of considering the backward heat equation on $ [t_{0}, T] \subset (0, +\infty) $, we can also consider the forward heat equation $ \partial_{t} u - \Delta_{x} u = 0, u : \mathbb{R}_{t}^{-} \times \mathbb{R}_{x}^{n} \rightarrow \mathbb{R} $ on some interval $ [T, t_{0}] \subset (-\infty, 0) $, where $ t_{0} < 0 $. We would like to translate this forward heat equation to some similar wave equation by embedding $ \mathbb{R}_{t}^{-} \times \mathbb{R}_{x}^{n} $ into $ \mathbb{R}_{t}^{-} \times \mathbb{R}_{x}^{n} \times \mathbb{R}_{y}^{N} $ with the Riemannian metric $ dy^{2} + dx^{2} $ again. \\
\\
Analogous to the backward heat equation, we consider the relation between time and $ y $ - variables by using $ t = -\frac{\lvert y \rvert^{2}}{2N} $ in this case. With this change of variables, we observe that
\begin{equation*}
dy^{2} + dx^{2} = dr^{2} + r^{2} d\omega_{1}^{2} + dx^{2} = -\frac{N}{2t} dt^{2} + (-t)d\omega_{1 \slash 2N}^{2} + dx^{2}.
\end{equation*}
\noindent Similarly, we obtain that
\begin{equation*}
\Delta_{y} = -\frac{\partial}{\partial t} - \frac{2t}{N} \cdot \frac{\partial^{2}}{\partial t^{2}}.
\end{equation*}
\noindent If $ u $ solves the forward heat equation $ (\partial_{t} - \Delta_{x})u = 0 $ on $ [T, t_{0}] \subset (-\infty, 0) $ we then obtain that
\begin{equation*}
(- \frac{2t}{N} \cdot \frac{\partial^{2}}{\partial t^{2}} - \Delta_{x} - \Delta_{y}) u = 0.
\end{equation*}
\noindent Note that due to the sign and domain that $ t $ is sitting in, the operator $ P^{(N)} : = - \frac{2t}{N} \cdot \frac{\partial^{2}}{\partial t^{2}} - \Delta_{x} - \Delta_{y} $ is still strictly
positive hyperbolic operator on $ [T, t_{0}] \times \mathbb{R}_{x}^{n} \times \mathbb{R}_{y}^{N} $. For the forward heat equation on negative time axis, we have the following wave equation approximation.
\begin{theorem}
Given the Cauchy problem of variable coefficient wave equation
\begin{equation}\label{one-and-two}
 (-\frac{2t}{N} \cdot \frac{\partial^{2}}{\partial t^{2}} - \Delta_{x, y} ) u = 0, (t, x, y) \in [T, t_{0}] \times \mathbb{R}_{x}^{n} \times \mathbb{R}_{y}^{N}, u(T, x, y) = h(x) -\infty < T < t_{0} < 0. 
\end{equation}
\noindent provided that $ h(x) \in \mathcal{C}_{0}^{\infty}(\mathbb{R}^{n}) $. \\
Then for each $ N $, there exists a radial solution $ u = u^{(N)}(t, x, y) $ of this Cauchy problem with respect to $ y $ - coordinates, i.e. $ u^{(N)}(t, x, y) = u^{(N)}(t, x, Ay), \forall A \in SO(N) $, such that the limit $ u(t, x) : = \lim_{N \rightarrow \infty} u^{(N)}(t, x, y) |_{ \lbrace t = -\lvert y \rvert^{2} \slash 2N \rbrace} $ exists at all points in $ [T, t_{0}] \times \mathbb{R}_{x}^{N} $. \\
Furthermore, the limit $ u(t, x) \in \mathcal{C}^{\infty}([T, t_{0}] \times \mathbb{R}_{x}^{n}) $ solves the forward heat equation
\begin{equation}
\partial_{t} u - \Delta_{x} u = 0, u(T, x) = h(x), (t, x) \in [T, t_{0}] \times \mathbb{R}_{x}^{n}.
\end{equation}
\end{theorem}
\begin{proof}
The proof is essentially the same as the proof of Theorem 4.1. The only difference we would like to point out here is that Theorem 4.2, which corresponds to the existence of the wave equation also holds when we replace the time interval $ [0, T] $ to $ [-T, 0] $, since the solution of wave equation is time-reversible.
\end{proof}
\begin{remark} In the Euclidean case, we introduce the notion of $ WF_{\infty} $, discuss explicitly the propagation of singularities for wave equation and show that $ WF_{\infty} $ contains exactly the points for which the heat kernel is not smooth. We now generalize the discussion in Euclidean case to Riemannian manifold case with Ricci flow to show that (\ref{heat}) can be approximated by (\ref{one}) strongly outside the wavefront set. We will also show how the wavefront set in Riemannian manifold case looks like.
\end{remark}

\noindent \section{The Analysis of Wave Approximation of Backward Heat Equation on Ricci Flow: Part II}
In this section, we would like to mimic the procedure in the Euclidean case to show the relation between the variable coefficient wave equation on open manifold $ M_{x} \times \mathbb{R}_{y}^{N} $ and the heat equation on $ M $, on some positive time interval $ [t_{0}, T] \subset [0, \infty) $. \\
\\
The difference here is that on manifold situation, the variable coefficient wave equation only admits a semi-global solution, as H\"ormander discussed in \cite{Hor3}. The notion of strictly hyperbolic operator on functions with manifold domain are also slightly different, which is a generalization of Euclidean case.

\begin{definition} {\bf \cite{Hor3}, Definition 23.2.3.} A differential operator $ P $ of order $ m $ in the $ \mathcal{C}^{\infty} $ manifold $ X $, with principal symbol $ p $, is said to be the strictly hyperbolic with respect to the level surfaces of $ \phi \in \mathcal{C}^{\infty}(X, \mathbb{R}) $ if $ p(x, \phi^{'}(x)) \neq 0 $ and the characteristic equation $ p(x, \xi + \tau \phi^{'}(x)) = 0 $ has $ m $ different real roots $ \tau $ for every $ x \in X $ and $ \xi \in T_{x}^{*} \backslash \mathbb{R} \phi^{'}(x) $.
\end{definition}
\begin{remark} H\"ormander stated in Section 23.4 of \cite{Hor3} that locally if $ \phi(x) = x_{n} $ and $ X $ is a neighborhood of $ \mathbb{R}^{n} $, the condition of the operator $ P $ of order $ m $ with principal symbol $ p $ is said to be strictly hyperbolic with respect to the level surfaces of $ \phi $ is equivalent to the following conditions: \\
(i) $ p(x, \xi) = \sum_{0}^{m} p_{j}(x, \xi^{'}) \xi_{n}^{j}, p_{m} = 1 $, where $ \xi^{'} \in \mathbb{R}^{n-1} $; \\
(ii) when $ x \in X $, $ x_{n} > 0 $, and $ \xi^{'} \in \mathbb{R}^{n-1} \backslash 0 $, the equation $ p(x, \xi) $ has only simple real roots $ \xi_{n} $.
\end{remark}
\noindent With this defition, H\"ormander pointed out a version of semi-global solution with respect to $ Pu = f $, as the following theorem shows.

\begin{theorem} {\bf \cite{Hor3}, Theorem 23.2.4.}
Let $ P $ be a differential operator of order $ m $ with $ \mathcal{C}^{\infty} $ coefficients in the $ \mathcal{C}^{\infty} $ manifold $ X $, and let $ Y \Subset X $ be an open subset. Assume that $ P $ is strictly hyperbolic with respect to the level surfaces of $ \phi \in \mathcal{C}^{\infty}(X, \mathbb{R}) $ and set
\begin{equation*}
X_{+} = \lbrace x \in X : \phi(x) > 0 \rbrace, X_{0} = \lbrace x \in X : \phi(x) = 0 \rbrace.
\end{equation*}
\noindent If $ f \in H_{loc}^{s}(X) $ has support in the closure of $ X_{+} $ one can then find $ u \in H_{loc}^{s + m - 1} (X) $ with support in the closure of $ X_{+} $ such that $ Pu = f $ in $ Y $. If $ s \geqslant 0 $, $ v $ is a vector field with $ v\phi = 1 $, and $ f \in \bar{H}_{loc}^{s}(X_{+}) $, $ \phi_{j} \in H_{loc}^{s + m - 1 - j}(X_{0}), j < m $, then there is some $ u \in H_{loc}^{s + m - 1}(X_{+}) $ such that $ Pu = f $ in $ X_{+} \cap Y $ and $ v^{j}u = \phi_{j} $ in $ X_{0} \cap Y $ when $ j < m $.
\end{theorem}
\noindent The semi-global solution of strictly hyperbolic operator restricts our discussion on some open, precompact submanifold of the space $ [t_{0}, T] \times M_{x} \times \mathbb{R}_{y}^{N} $. To do this, we choose an open, precompact submanifold $ \tilde{M}_{x} \times B_{y}^{N} $ of $ M_{x} \times \mathbb{R}_{y}^{N} $ and consider the operator $ \frac{2t}{N} \partial_{t}^{2} - \Delta_{x, g(t)} - \Delta_{y} $ on $ [t_{0}, T] \times \tilde{M}_{x} \times B_{y}^{N} $. We also require the boundary of the closure of $ M_{x} $, i.e. $ \partial \bar{\tilde{M}}_{x} $ to be at least $ \mathcal{C}^{1} $. \\
\\
We now prove Theorem 3.1, which is the wave equation approximation of the backward heat equation on Riemannian manifold. To show this, we check what is the real expression in (\ref{one}) by analyzing the Laplacian $ \tilde{\Delta}_{\tilde{g}^{(N)}(t)} $ on $ \tilde{M}^{(N)} : = M_{x} \times \mathbb{R}_{y}^{N} $. Here $ (M, g(t)) $ is a noncompact manifold without boundary and
\begin{equation*}
\tilde{g}^{(N)} = g(t) \oplus dy^{2}, y = (y_{1}, \dotso, y_{N}),
\end{equation*}
\noindent with the corresponding volume form $ d\mu(t) \cdot dy $. We compute $ \tilde{\Delta}_{\tilde{g}^{(N)}(t)} $ in (\ref{thirteen}) by Dirichlet form:
\begin{align*}
\int_{\tilde{M}^{(N)}} u  \tilde{\Delta}_{\tilde{g}^{(N)}(t)} vd\mu(t) dy & = -\int_{\tilde{M}^{(N)}} \tilde{g}^{(N)}(\nabla^{(N)}u, \nabla^{(N)} v) d\mu(t) dy \\
& = -\int_{\mathbb{R}_{y}^{N}} \int_{M_{x}} \nabla_{y} u \cdot \nabla_{y} v + g(t)(\nabla_{x, g(t)}u, \nabla_{x, g(t)}v) d\mu(t) dy \\
& = \int_{\tilde{M}^{(N)}} u (\Delta_{y} v +  \Delta_{x, g(t)} v - \frac{r}{2N} \text{tr}(\dot{g}) \partial_{t} v) d\mu(t) dy,
\end{align*}
\noindent by integration by parts. Under the existence of the Ricci flow at desired time, the above formula indicates that $  \tilde{\Delta}_{\tilde{g}^{(N)}(t)} $ is of the form (\ref{thirteen}). With (\ref{thirteen}), we observe that (\ref{fourteen}) is exactly the wave equation of the form
\begin{equation}\label{twentysix}
\frac{2t}{N} \frac{\partial^{2}u}{\partial t^{2}} - \Delta_{x, g(t)}u - \Delta_{y} u = 0.
\end{equation}
\noindent According to (\ref{twentysix}), we can start at the proof of Theorem 3.1.
\begin{proof} {\bf Proof of Theorem 3.1.} 
First of all, we should show that the wave equation (\ref{one}) has a solution on $ Y $. According to the discussion above, (\ref{one}) is equivalent as (\ref{twentysix}), so we show the solvability of (\ref{twentysix}) on $ Y $. \\
\\
To show this, we check the strict hyperbolicity in the sense of Definition 5.1, and then apply Theorem 5.1 to conclude the existence of solution on $ Y $. \\
\\
Note that $ Y = (t_{0}, T) \times \tilde{M}_{x, g(t)} \times \mathbb{B}_{y}^{N} $, to check the strict hyperbolicity of $ P $ with respect to $ \phi $, which is the defining function of $ \partial (\bar{Y}) $, we check the strict hyperbolicity near the boundary of $ \bar{Y} $ by Remark 5.1. It suffices to check strict hyperbolicity with the defining function of the form $ \phi(t, x, y) = t - t_{0} $. Since $ [t_{0}, T] $ is a manifold with single chart, the expression of $ \phi $ is uniform with respect to all charts that consists of the covering of $ Y $. \\
\\
As before, we consider the following equivalent PDE:
\begin{equation*}
(\frac{2t}{N} \frac{\partial^{2}}{\partial t^{2}} - \Delta_{x, g(t)} - \Delta_{y}) u = R \Leftrightarrow (D_{t}^{2} + \frac{N}{2t} \Delta_{x, g(t)} + \frac{N}{2t}\Delta_{y}) u = - \frac{N}{2t}R.
\end{equation*}
\noindent The solvability for above two equations are exactly the same since $ t \in [t_{0}, T], t_{0} > 0 $. The corresponding differential operator for the second equation above is defined to be
\begin{equation*}
\tilde{P}^{(N)} = D_{t}^{2} + \frac{N}{2t} \Delta_{x, g(t)} + \frac{N}{2t}\Delta_{y}.
\end{equation*}
\noindent Within any smooth chart of $ Y $, the principal symbol of $ \tilde{P}^{(N)} $, denoted to be $ \tilde{p}^{(N)} $ is given to be
\begin{equation*}
\tilde{p}^{(N)} = \tau^{2} - \frac{N}{2t}(\lvert \xi \rvert_{g(t)}^{2} + \lvert \eta \rvert^{2}).
\end{equation*}
\noindent Here $ \tau \in T_{t}^{*} \mathbb{R} $, $ \xi \in T_{x}^{*} M_{x, g(t)} $ and $ \eta \in T_{y}^{*} \mathbb{R}_{y}^{N} $. Since $ \phi(t, x, y) = t - t_{0} $, we can translate the time axis and consider the function $ \phi = t $ with the principal symbol
\begin{equation*}
\tilde{p}^{(N)} = \tau^{2} - \frac{N}{2(t + t_{0})}(\lvert \xi \rvert_{g(t)}^{2} + \lvert \eta \rvert^{2}), t \in [0, T - t_{0}].
\end{equation*}
\noindent According to Remark 5.2, it is clear that $ p_{2} = 1 $ and $ \tilde{p}^{(N)} = 0 $ has simple real roots. \\
\\
By checking Definition 5.1, we observe that $ \phi^{'}(t, x, y) = (1, 0, \dotso, 0) $ and it follows that
\begin{equation*}
\tilde{p}^{(N)}(t, x, y, \phi^{'}(x))  = 1 \neq 0.
\end{equation*}
\noindent Furthermore, we observe that locally within some single smooth chart, the local form of principal symbol has
\begin{equation*}
\tilde{p}^{(N)}(t, x, y, \tilde{\tau} + \tau, \xi, \eta) = (\tilde{\tau} + \tau)^{2} - \frac{N}{2t} (\lvert \xi \rvert_{g(t)}^{2} + \lvert \eta \rvert^{2}).
\end{equation*}
\noindent And $ \tilde{p}^{(N)}(t, x, y, \tilde{\tau} + \tau, \xi, \eta) = 0 $ has two different real roots $ \tilde{\tau} $ for fixed tuple $ (t, x, y, \tau, \xi, \eta) $. \\
\\
Since the operator $ \tilde{P}^{(N)} $ is strictly hyperbolic with respect to $ \phi = t $, it follows that the initial value problem of (\ref{one}) has at least one solution on $ Y $ in the weak sense. \\
\\
The following procedure is almost the same as in the Euclidean case. First we show that (\ref{twentysix}) has a strong solution outside the wavefront set. Then we construct a radial solution of (\ref{twentysix}) in $ y $ and analyze its wavefront set. We then show that on the hypersurface $ t = \frac{\lvert y \rvert^{2}}{2N} $ the projection of this radial solution onto $ \mathbb{R}_{t} \times M_{x} $ does solve the corresponding heat equation (\ref{heat}), again outside the wavefront set. Lastly we show that the limit $ u(t, x) = \lim_{N \rightarrow \infty} u^{(N)}(t, x, y) $ exists pointwisely outside $ WF_{\infty} $ and solves (\ref{heat}). In addition, the initial condition is satisfied. \\
\\
By Theorem 4.2, we conclude that for each $ N $, the equation (\ref{one}), or equivalently (\ref{twentysix}) has a strong solution $ \tilde{u}^{(N)} $ outside $ WF(\tilde{u}^{(N)}) $, provided that the Ricci flow exists at time $ T $ since on $ [t_{0}, T] \subset (0, \infty) $, the principal symbol of the operator
\begin{equation*}
P^{(N)} = \frac{2t}{N} \frac{\partial^{2}}{\partial t^{2}} - \Delta_{x, g(t)} - \Delta_{y} \Rightarrow p^{(N)} : = \sigma(P^{(N)}) = \frac{2t}{N} \tau^{2} - \lvert \xi \rvert_{g(t)}^{2} - \lvert \eta \rvert^{2},
\end{equation*}
\noindent is of real type, properly supported, and is a strictly positive hyperbolic operator, due to the same argument in proof of Theorem 4.1. \\
\\
We check that if $ \tilde{u}^{(N)}(t, x, y) $ solves (\ref{twentysix}) , then so is $ u^{(N)}(t, x, Ay) $ for all $ A \in SO(N) $. We can see that in every local coordinate patch with some coordinates representation:
\begin{align*}
& \text{ $ \tilde{u}^{(N)}(t, x, y) $ solves (\ref{twentysix})} \Rightarrow \frac{2t}{N} \partial_{tt} \tilde{u}^{(N)}(t, x, y) - \Delta_{x, g(t)}\tilde{u}^{(N)}(t, x, y) - \Delta_{y}\tilde{u}^{(N)}(t, x, y) = 0 \\
\Rightarrow & \frac{2t}{N} \partial_{tt} \tilde{u}^{(N)}(t, x, Ay) - \Delta_{x, g(t)}\tilde{u}^{(N)}(t, x, Ay) - \Delta_{Ay}\tilde{u}^{(N)}(t, x, Ay) = 0, \text{$ y $ is a dummy variable} \\
\Rightarrow & \frac{2t}{N} \partial_{tt} \tilde{u}^{(N)}(t, x, Ay) - \Delta_{x, g(t)}\tilde{u}^{(N)}(t, x, Ay) - \Delta_{y}\tilde{u}^{(N)}(t, x, Ay) = 0, \text{by Proposition 4.1}.
\end{align*}
\noindent Hence we conclude that $ \tilde{u}^{(N)}(t, x, Ay) $ solves (\ref{twentysix}), for all $ A \in SO(N) $, by Proposition 4.1. We now define a rotationally invariant solution of (\ref{twentysix}) as in previous proof. We define locally within each patch that
\begin{equation}\label{twentyseven}
u^{(N)}(t, x, y) = \frac{1}{\text{Vol}(SO(N))}\int_{SO(N)} \tilde{u}^{(N)}(t, x, Ay) d\mu_{N}.
\end{equation}
\noindent Where $ d\mu_{N} $ is the volume form on $ SO(N) $. It works equally well globally on $ \tilde{M} \times [t_{0}, T] $. By the same reason as in proof of Theorem 4.1, we conclude that $ u^{(N)} $ solves (\ref{twentysix}) outside $ \bigcup WF(u^{(N)}) $. Of course, $ u^{(N)} $ is constant along the hypersurface $ t = \frac{\lvert y \rvert^{2}}{2N} $. \\
\\
We now analyze the wavefront set $ WF(\tilde{u}^{(N)}) $ as well as $ WF(u^{(N)}) $. We claim that
\begin{equation}\label{twentyeight}
WF(u^{(N)}) = WF(\tilde{u}^{(N)}) \subset \Sigma(P^{(N)}).
\end{equation}
\noindent Note that we can only do this analysis locally within each coordinate patch. Within each coordinate patch, the Riemannian metric is given of the form
\begin{equation*}
\tilde{g}^{(N)} = g_{ij}dx^{i}dx^{j} + \sum_{k} (dy^{k})^{2}, g_{ij} = g_{ij}(t, x).
\end{equation*}
\noindent Here we use Einstein summation convention. Hence the principal symbol of $ P^{(N)} $ has local expression
\begin{equation*}
p^{(N)} = \frac{2t}{N} \tau^{2} - g^{ij}\xi_{i}\xi_{j} - \lvert \eta \rvert^{2}.
\end{equation*}
\noindent Note that $ g^{ij} = g^{ij}(t, x) $. It follows that the Hamiltonian for $ p^{(N)} $ locally is of the form
\begin{equation}\label{twentynine}
H_{p^{(N)}} = \frac{4t \tau}{N} \frac{\partial}{\partial t} - 2g^{ij} \xi_{j} \frac{\partial}{\partial x_{i}} - \sum_{k} 2 \eta_{k} \frac{\partial}{\partial y_{k}} - \left(\frac{2\tau^{2}}{N} - \frac{\partial g^{ij}}{\partial t} \xi_{i} \xi_{j}\right) \frac{\partial}{\partial \tau} - \frac{\partial g^{ij}}{\partial x^{k}} \xi_{i} \xi_{j} \frac{\partial}{\partial \xi_{k}}.
\end{equation}
\noindent From (\ref{twentynine}), we can see that the null bicharacteristics with $ (t(s), x(s), y(s), \tau(s), \xi(s), \eta(s)) $ locally must satisfy the following first order ODE system
\begin{align}\label{thirty}
& \dot{t} = \frac{4t \tau}{N}, \dot{x}^{i} = - 2g^{ij} \xi_{j}, \dot{y}^{k} = -2 \eta_{k}, \dot{\eta}_{k} = 0, \nonumber \\
& \dot{\tau} = - \left(\frac{2\tau^{2}}{N} - \frac{\partial g^{ij}}{\partial t} \xi_{i} \xi_{j}\right), \dot{\xi}_{k} = - \frac{\partial g^{ij}}{\partial x^{k}} \xi_{i} \xi_{j}.
\end{align}
\noindent From this, an immediate consequence is that
\begin{equation*}
\dot{\xi}_{K} = -\frac{\partial g_{ij}}{\partial x^{k}} \dot{x}^{i} \dot{x}^{j}.
\end{equation*}
\noindent Within the characteristic set, locally we must have $ p^{(N)} = 0 $ which implies that $ \eta(0) \in WF(\tilde{u}^{(N)}) $ if and only if $ A \eta(0) \in  WF(\tilde{u}^{(N)}) $. Partially solving the integral curves above, we have
\begin{equation*}
\eta_{k} = \eta_{k}(0), y^{k} = y^{k}(0) - 2\eta^{k}(0) s.
\end{equation*}
\noindent This part is the same as in Theorem 4.1, and hence we conclude that (\ref{twentyeight}) holds. In addition, the propagation of singularities of $ u^{(N)} $ for each $ N $ are along integral curves satisfying (\ref{thirty}). \\
\\
Since $ u^{(N)} $ is radial in $ y $ and solves (\ref{twentysix}), we see that $ \Delta_{y} $ only corresponds to radial variation only. Along the hypersurface $ t = \frac{ \lvert y \rvert^{2}}{2N} $, we have $ u^{(N)}(t, x, y) = u^{(N)}(t, x, \lvert y \rvert) = u^{(N)}(t, x, \sqrt{2Nt}) $, and hence the projection of solution $ u^{(N)} $ onto $ [t_{0}, T] \times M_{x} $ in each coordinate patch must satisfy
\begin{align*}
0 & = \left(\frac{2t}{N} \cdot \frac{\partial^{2}}{\partial t^{2}} - \Delta_{x, g(t)} \right) u^{(N)}(t, x, \lvert y \rvert) - \sum_{j=1}^{N} \frac{\partial^{2}u^{(N)}(t, x, \lvert y \rvert)}{\partial y_{j}^{2}} \\
& = \left(\frac{\sum_{i=1}^{N} y_{i}^{2}}{N^{2}} \cdot \frac{\partial^{2}}{\partial t^{2}} - \Delta_{x, g(t)} ) u^{(N)}(t, x, \sqrt{2Nt}\right) - \sum_{j=1}^{N} \left(\frac{1}{N} \frac{\partial }{\partial t} + \frac{y_{j}^{2}}{N^{2}} \frac{\partial^{2} }{\partial t^{2}}\right)u^{(N)}(t, x, \sqrt{2Nt}) \\
& = - \Delta_{x, g(t)} u^{(N)}(t, x, \sqrt{2Nt}) - \frac{\partial u^{(N)}(t, x, \sqrt{2Nt})}{\partial t}.
\end{align*}
\noindent It follows that the projection of solution $ u^{(N)} $ of (\ref{twentysix}), or equivalently, (\ref{one}), at the hypersurface $ t = \frac{\lvert y \rvert^{2}}{2N} $ onto $ [t_{0}, T] \times M_{x} $ solves (\ref{heat}), while $ (t, x) $ is outside the projection of $ WF(u^{(N)}) $ onto $ [t_{0}, T] \times M_{x} $. In addition, $ u(t, x) : = \lim_{N \rightarrow \infty} u^{(N)}(t, x, y) $ solves (\ref{heat}) provide that $ (t, x) $ is outside the projection of $ WF_{\infty} $ onto $ [t_{0}, T] \times M_{x} $ in every coordinate patch. \\
\\
Lastly, it is trivial to see that the initial condition is satisfied, due to the argument above and uniqueness of solution of (\ref{heat}).
\end{proof}

\noindent Analogous to the Euclidean case, we can extend the time interval $ [t_{0}, T] $ to $ (0, T] $ for arbitrarily large $ T $. We can also consider the forward heat equation and its corresponding wave approximation on Riemannian manifold.
\begin{theorem}
Let $ Y : = (T, T_{0}) \times \tilde{M}_{x, g(t)} \times B_{y}^{N} $ be any open, precompact submanifold of $ X : = [t_{0}, T] \times M_{x, g(t)} \times \mathbb{R}_{y}^{N} $ and the boundary of the closure of $ Y $ is $ \mathcal{C}^{1} $. Assume that Ricci flow exists at some time $ T > 0 $ and $ \text{supp}(R(t, x)) \in Y $ where $ R $ is the scalar curvature. \\
Given the Cauchy problem of variable coefficient wave equation
\begin{align}
& \left(-\frac{2t}{N} \cdot \frac{\partial^{2}}{\partial t^{2}} + \frac{rR(t, x)}{N} \frac{\partial}{\partial t}- \tilde{\Delta}_{\tilde{g}^{(N)}} \right) u = 0, (t, x, y) \in [T, t_{0}] \times M_{x, g(t)} \times (\mathbb{R}_{y}^{N} \backslash \lbrace 0 \rbrace) \\
& u(T, x, y) = h(x), -\infty < T < t_{0} < 0. \nonumber
\end{align}
\noindent provided that $ h(x) \in \mathcal{C}_{0}^{\infty}(M_{x}) $. \\
(i) For each $ N $, there exists a radial solution $ u = u^{(N)}(t, x, y) $ of this Cauchy problem with respect to $ y $ - coordinates, i.e. $ u^{(N)}(t, x, y) = u^{(N)}(t, x, Ay), \forall A \in SO(N) $, such that the limit $ u(t, x) : = \lim_{N \rightarrow \infty} u^{(N)}(t, x, y) |_{ \lbrace t = -\lvert y \rvert^{2} \slash 2N \rbrace} $ exists outside points $ \mathcal{P} $ in the projection of wavefront set $ WF_{\infty} $ onto $ [T, T_{0}] \times \tilde{M}_{x, g(t)} $. \\
(ii) Outsider points in $ \mathcal{P} $, the limit $ u(t, x) \in [T, t_{0}] \times \tilde{M}_{x, g(t)} $ solves the forward heat equation
\begin{equation}
\partial_{t} u - \Delta_{x, g(t)} u = 0, u(T, x) = h(x), (t, x) \in [T, t_{0}] \times M_{x, g(t)}.
\end{equation}
\end{theorem}
\begin{proof} The proof is essentially the same as the proof of Theorem 3.1, since the wave equation is time reversible and hence we just set $ \phi(t, x, y) = t_{0} - t $ and everything else is essentially the same.
\end{proof}

\section{Further Discussion}
As we can see, the equation, e.g., (\ref{one}), has singularity when $ t = 0 $, we may fix this problem by replacing $ \frac{2t}{N} $ by $ \frac{2t}{N} + c $ where $ c $ is some positive number. Recall (\ref{ten}), (\ref{eleven}), we anticipate that with this setup, the coefficient ahead of $ \partial_{tt} $ is of the form $ \frac{2t}{N} + R(t) $. If $ R(0) > 0 $, then it follows that the operator
\begin{equation*}
\left(\frac{2t}{N} + R(t)\right)\partial_{tt} - \Delta_{x, g(t)} - \Delta_{y},
\end{equation*}
\noindent is the strictly hyperbolic operator starting at time $ t = 0 $ and hence we could discuss the solution of wave equation on $ \mathbb{R}_{\geqslant 0} \times M_{x, g(t)} \times \mathbb{R}_{y}^{N} $. We did a simplification of (\ref{ten}), saying that we consider the Riemannian metric on $ \mathbb{R}_{t} \times \tilde{M} $ as
\begin{equation}\label{thirtyone}
\tilde{g}^{(N)} = dy^{2} + c \cdot \frac{r^{2}}{N^{2}}dr^{2} + g(t), r^{2} = \sum_{i=1}^{N} y_{i}^{2}.
\end{equation}
\noindent Here $ c > 0 $ is some positive constant. From (\ref{thirtyone}) and the same argument as in Section 2, we conclude with the relation $ t = \frac{\lvert y \rvert^{2}}{2N} $ that
\begin{equation}\label{thirtytwo}
\tilde{g}^{(N)} = \left(\frac{2t}{N} + c\right) dt^{2} + td\omega_{1 \slash 2N}^{2} + g(t).
\end{equation}
\noindent We now compute the expression of $ \tilde{\Delta}_{\tilde{g}^{(N)}} $ with respect to the Riemannian metric (\ref{thirtyone}). \\
\\
I will compute $ \tilde{\Delta} $ and then we have the following theorem. \\
\\

We study the Cauchy problem of the following variable coefficient wave equation.
\begin{align}\label{thirtythree}
&\left(\left(\frac{2t}{N} + c\right) \cdot \frac{\partial^{2}}{\partial t^{2}} + \frac{t\cdot \text{tr}(\dot{g})}{2N} \frac{\partial}{\partial t} + \text{some lower order derivatives in $ y $}- \Delta_{\tilde{g}^{(N)}(t)} \right) u = 0, \\
& (t, x, y) \in [0, T] \times M_{x, g(t)} \times \mathbb{R}_{y}^{N}, u(T, x, y) = h(x). \nonumber
\end{align}
\noindent provided that $ h(x) \in \mathcal{C}_{0}^{\infty}(M_{x, g(t)}) $. Here $ M_{g(t)} $ is the noncompact Riemannian manifold $ (M, g(t)) $ without boundary satisfying Ricci flow $ \dot{g} = - 2Ric $; $ R(t, x) $ is the scalar curvature of $ (M, g(t)) $ at $ x \in M $; and the Riemannian metric $ \tilde{g}^{(N)}(t) $ on $  M_{x, g(t)} \times \mathbb{R}_{y}^{N} $ is given by (\ref{thirtyone}).
\begin{theorem}
Let $ 0 < t_{0} < T < \infty $. Let $ Y : = (t_{0}, T) \times \tilde{M}_{x, g(t)} \times B_{y}^{N} $ be any open, precompact submanifold of $ X : = [t_{0}, T] \times M_{x, g(t)} \times \mathbb{R}_{y}^{N} $ and the boundary of the closure of $ Y $ is $ \mathcal{C}^{1} $. Assume that Ricci flow exists at some time $ T > 0 $ and $ \text{supp}(R(t, x)) \in Y $ where $ R $ is the scalar curvature. \\
(i) for each $ N $, there exists a radial solution $ u = u^{(N)}(t, x, y) $ of (\ref{thirtythree}) with respect to $ y $ - coordinates, i.e. $ u^{(N)}(t, x, y) = u^{(N)}(t, x, Ay), \forall A \in SO(N) $, such that the limit $ u(t, x) : = \lim_{N \rightarrow \infty} u^{(N)}(t, x, y) |_{ \lbrace t = \lvert y \rvert^{2} \slash 2N \rbrace} $ on $ [t_{0}, T] \times \tilde{M}_{x, g(t)} $ exists outside points $ \mathcal{P} $ in the projection of wavefront set $ WF_{\infty} $ onto $ [0, T] \times \tilde{M}_{x, g(t)} $. 
\par
(ii) Outside points in $ \mathcal{P} $, the limit $ u(t, x) \in [0, T] \times \tilde{M}_{x, g(t)} $ solves the backward heat equation below
\begin{equation}\label{thirtyfour}
\partial_{t} u + \Delta_{x, g(t)} u = 0, u(T, x) = h(x), (t, x) \in [t_{0}, T] \times M_{x, g(t)}.
\end{equation}
\end{theorem}
\noindent To check the validity of this setup, we could reduce our manifold $ M_{x, g(t)} $ to be the Euclidean space $ \mathbb{R}_{x}^{n} $, in which case the scalar curvature is zero. In that case, the equation (\ref{thirtythree}) reduced to the homogeneous wave equation, and all lower terms are gone. \\
\\
Similarly as above, we can apply the same method to analyze the behavior of the solution of forward heat equation. \\
\\
Changing the metric is equivalent to changing the intrinsic geometric structure, we can also change the topological structure of the space we affiliate. For example, the Euclidean space wi the metric (\ref{thirtyone}) corresponds to some spherical structure. We may affiliate the upper half plane model with negative curvature when $ M $ has some negative curvature features. \\
\\
{\bf Acknowledgement} \\
The author would like to thank his advisor Prof. Steven Rosenberg for his support and mentorship.

\bibliographystyle{plain}
\bibliography{Wave}

\end{document}